\documentclass[11pt]{article} %
\usepackage[titletoc,title]{appendix}
\usepackage{amsmath}
\usepackage{amsfonts}
\usepackage{amssymb}
\usepackage{epsfig}
\usepackage[font=normalsize]{subfig}
\usepackage{amscd}
\usepackage{latexsym}
\usepackage[colorlinks=black,linkcolor=red,anchorcolor=blue,citecolor=green]{hyperref}
\usepackage{CJK}
\usepackage{color}
\usepackage{leftidx}
\usepackage{mathdots}
\usepackage{multirow}
 
\newtheorem{theorem}{\bf Theorem}[section]
\newtheorem{lemma}[theorem]{\bf Lemma}
\newtheorem{proposition}[theorem]{\bf Proposition}

\newtheorem{example}[theorem]{\bf Example}

\newenvironment{proof}{\noindent{\em Proof:}}{\quad \hfill$\Box$\vspace{2ex}}
\newenvironment{proofspecial}{\noindent{\em Proof of Lemma 2.1:}}{\quad \hfill$\Box$\vspace{2ex}}

\newenvironment{remark}{\noindent{\bf Remark}}{\vspace{2ex}}
\makeatletter
\@addtoreset{equation}{section}
\makeatother
\renewcommand{\theequation}{\arabic{section}.\arabic{equation}}

\DeclareMathOperator{\diag}{diag}

\def \aR {\mathbb R}

\def \bL {\mathcal{L}}

\def \bO {\mathcal{O}}

\def \bD {\mathcal{D}}

\def \bW {\mathcal{W}}

\def \BA {\textbf{A}}
\def \BB {\textbf{B}}

\def \Bf {\textbf{f}}

\def \Bx {\textbf{x}}

\def \Bu {\textbf{u}}

\def \Bc {\textbf{c}}
\def \Bq {\textbf{q}}

\def \Br {\textbf{r}}

\begin{document}
\begin{CJK}{GBK}{song}

\title{\bf  Levin methods for highly oscillatory integrals with singularities}
\author{ Yinkun Wang
         \footnotemark[1],
         Shuhuang Xiang
         \footnotemark[2]
         \footnotemark[3]
         }

\renewcommand{\thefootnote}{\fnsymbol{footnote}}

\footnotetext[1]{Department of Mathematics, National University of Defense Technology, Changsha, Hunan, P. R. China}
\footnotetext[2]{School of Mathematics and Statistics, Central South University, Changsha, Hunan, P. R. China.}
\footnotetext[3]{Correspondence author: xiangsh@mail.csu.edu.cn}
\date{}
\maketitle{}
\begin{abstract}
{\color{black}
In this paper, new Levin methods are presented for calculating oscillatory integrals with algebraic and/or logarithmic singularities.
}
{\color{black}
To avoid singularity,} the technique of singularity separation is applied and then the singular ODE 
{\color{black}
occurring} in classic Levin methods is converted into two kinds of non-singular ODEs. 
{\color{black}
The solutions of one can be obtained explicitly, while those of the other can be solved efficiently by collocation methods.}
The proposed methods can attach arbitrarily high asymptotic orders and also enjoy superalgebraic convergence with respect to the number of collocation points. 
Several numerical experiments are presented to validate the efficiency of the proposed methods.
\end{abstract}


\section{Introduction}

The computation of $\int_a^b f(x)e^{i w g(x)}dx$
 occurs in a wide range of practical problems
and applications,
{\color{black}
 e.g.,} nonlinear optics, fluid dynamics, computerized tomography,
celestial mechanics, electromagnetics, acoustic scattering, etc. 
The high oscillation ($|w|\gg 1$) means that classical Gaussian quadrature requires $\bO(w)$ quadrature points, which is impractical.

To handle the difficulty caused by rapid oscillation, many effective methods have been proposed for oscillatory integrals without 
{\color{black}
singularities, such} as Filon-type methods \cite{GAO2017,ISERLES2004,ISERLES2005N}, Levin methods \cite{LEVIN1982,OLVER2006}, the generalized quadrature rule \cite{EVANS2000}, 
{\color{black} and} numerical steepest-descent methods \cite{HUYBRECHES2006}. We refer interested readers to \cite{ISERLES2018B,HUYBRECHES2009} for a review of these methods.

However, in the context of electromagnetic and acoustic scattering, one frequently must compute many oscillatory integrals with singularities of the form 
 \begin{equation}\label{SOI}
I_{w}^{[0,a]}[f,s,g]:=\int_0^a f(x)s(x)e^{iwg(x)}dx
\end{equation}
(see \cite{BRUNO2004,CHANDLER2012,COLTON1983,ISERLES2018B,DOMINGUEZ2013,SPENCE2014}),
where $f$ and $g$ are suitably smooth functions, $g'(x)\neq 0, x\in[0,a]$, and $w$ is a real 
{\color{black} parameter, the absolute value of which could be extremely large.} Without loss of generality, we assume $g(0)=0$ and $g'(x)>0$ for $x\in[0,a]$. If $g(0)\neq 0$, we replace $g(x)$ by $g(x)-g(0)$, and if $g'(x)<0, x\in(0,a]$, the function $g(x)$ is replaced by $-g(x)$ and $w$ by $-w$, respectively.
The function $s$ is singular 
{\color{black}
and the singularity locates at $x=0$.} If the
integral has finitely many singular points, then it can
be rewritten in terms of integrals of the form $I_{w}^{[0,a]}[f,s,g]$.

{\color{black}
A significant amount of work} has also been done on the computation of singular and oscillatory integrals of the type \eqref{SOI}.
The asymptotic behavior for the integral was obtained by repeated integration by parts \cite{ERDELYI1955,GAO2016} or  by the inverse functions \cite{LYNESS2009}. When the oscillator is linear, i.e., $g(x)=x$, it was studied by the Clenshaw-Curtis-Filon-type methods \cite{KANG2011,KANG2013,XAINGGUO2014},
{\color{black}
 in which the modified moments} can be obtained numerically by stable recurrence relations.
However, 
{\color{black}
these methods} may not be suitable for the general case since it is difficult to 
{\color{black}
accurately calculate the modified moments  $\int_0^a T_j(x)s(x)e^{iwg(x)}dx$, where $T_j(x)$} denotes the shifted Chebyshev polynomial of the first kind of degree $j$.

 A composite Filon-Clenshaw-Curtis quadrature was proposed in \cite{DOMINGUEZ2013} 
 {\color{black}
 based on} efficient evaluation of the inverse function of the oscillator for the case of nonlinear oscillators.
Another kind of composite method was developed recently in \cite{MAXU2017} based on the careful design of meshes to achieve a convergence of the polynomial order or of the exponential order. 
{\color{black}
The main disadvantage of the composite methods is that sub-intervals near the singular point in the designed mesh have very small lengths and thus may cause serious round-off-error problems.}

Based on the numerical steepest method, Gauss-type quadrature has been used for computation of the highly oscillatory integrals with algebraic singularities with a linear
oscillator \cite{HE2014,XU2015,Xu2016}. There is still much  work  to do on the computation of the singular and oscillatory integrals, especially with complicated oscillators 
{\color{black} 
in terms of efficiency and accuracy.}

In this paper, we are interested in efficient numerical methods for (1.1) with
\[ s(x)=x^\alpha\;\text{or}\; x^\alpha\log x,\quad 0<|\alpha|<1,\]
and develop new efficient methods based upon the classic Levin method, which is quite different from the existing methods, to compute the integral of the type $I_{w}^{[0,a]}[f,s,g]$. For the specification, we set $s_1(x)=x^\alpha$ and  $s_2(x)=x^\alpha\log x$.

The spirit of 
{\color{black}
the Levin method} in the computation of the integral $I_w^{[0,a]}[f,s,g]$  is to find a function $p$ such that $\left(p(x)e^{iwx}\right)'=f(x)s(x)e^{iwx}$. This is 
{\color{black}
equivalent to obtaining} a particular solution of the ODE,
\begin{equation}\label{levinode}
\bL[p](x)\equiv p'(x)+iwp(x)=f(x)s(x).
\end{equation}
However, the particular solution of the ODE \eqref{levinode} 
{\color{black}
cannot} be obtained directly by collocation methods due to the singular forcing function. The singularity would cause large errors.

To deal with the singularity,  a technique of singularity separation is developed.  The computation of integral $I_{w}^{[0,a]}[f,s,g]$ can be converted into the solution of two kind of ODEs.
One kind of ODE has an explicit solution with the vanishing initial condition, while the other 
possesses a specific structure of the form
\begin{equation}\label{modelode}
iwg'(x)c_0+g(x)q_1'(x)+[1+\alpha+iwg(x)]g'(x)q_1(x)=f(x),
\end{equation}
where $f$ and $g$ are given functions,  and function $q_1$ and coefficient $c_0$ are unknown 
{\color{black}
and must be determined.}
It will be proved in this paper that there exists at least one pair solution of a non-oscillatory function $q_1$ and a number $c_0$ for \eqref{modelode}. The term \textit{non-oscillatory} is understood in the sense that the function's derivatives of high orders are independent of the frequency. 
This means that the  ODE \eqref{modelode} can be solved well by collocation methods without the influence of the high oscillation. A new collocation method is developed for 
{\color{black}
the ODE \eqref{modelode}} by adopting the differential matrix based on the Chebyshev-Gauss-Radau points. 
In particular, following the asymptotic method and convergence rates for Filon-type method in \cite{ISERLES2004,ISERLES2005N}, the convergence for the new Levin methods is derived.

We also show the equivalence between the new Levin methods and the corresponding Filon-type methods 
{\color{black}
with a proper basis.} The new methods for the oscillatory integrals with algebraic and/or logarithmic singularities can avoid 
{\color{black}
the round-off-error problem} caused by the tiny meshes and the computation of the modified moments, and also enjoy 
{\color{black}
the following merits.}
\begin{enumerate}
  \item They are applicable for 
  {\color{black} 
  nonlinear oscillators.}
  \item They converge 
  {\color{black}
  supralgebraically} with respect to the number of collocation points and the higher oscillation.
  \item Their asymptotic order with respect to the frequency is $\bO(w^{-s-1-\min\{1+\alpha,1\}})$ for algebraic singularities and  $\bO(\delta_{\alpha}(w)w^{-s-1-\min\{1+\alpha,1\}})$ for algebraic and logarithmic singularities, where $\delta_\alpha$ is defined in \eqref{delta1}.
\end{enumerate}

{\color{black}
The rest of this paper} is organized as follows. In Section 2, we develop a new Levin method for oscillatory integrals with algebraic singularity and then analyze the asymptotic order and the convergence. The equivalence between the new Levin method and the Filon-type method is studied. Another new Levin method is developed analogously for $I_{w}^{[0,a]}[f,s_2,g]$ in 
{\color{black}
Section 3}. We  construct the new collocation method for 
{\color{black}
the ODE \eqref{modelode} in Section 4. 
Numerical results are shown in Section 5 to validate the theory developed herein. }

\section{New Levin method for \texorpdfstring{$I_w^{[0,a]}[f,s_1,g]$}{I1}}

We commence from the integral $I_w^{[0,a]}[f,s_1,g]$ with algebraic singularity, assuming that the oscillator $g$ is strictly monotone in $[0,a]$. To cope with the singularity,  our basic idea is to seek a particular solution 
{\color{black}
the singularity of which is represented separately.}

Inasmuch as the observation that the solution of the corresponding ODE \eqref{levinode} possesses the algebraic singularity, a particular solution $p$  is assumed to have a specific form
 \[p(x)=q(x)g^{\alpha}(x)+h(x),\]
 where 
 {\color{black}
 the functions $q$ and $h$} need to be determined. The selection of $g^\alpha(x)$ instead of $x^\alpha$ is necessary, which will be seen later. 
 The substitution of $p$ in 
 {\color{black}
 the ODE \eqref{levinode}} leads to a new ODE for $q$ and $h$,
  \begin{equation}\label{gpeq}
 \begin{split}
\left (q'(x)+iw g'(x)q(x)\right)g^{\alpha}(x) & +h'(x)+iw g'(x)h(x)\\
 & +\alpha q(x)g'(x)g^{\alpha-1}(x)=f(x)x^{\alpha}.
 \end{split}
 \end{equation} 
{\color{black}
A new function is defined,}
\begin{equation}\label{f1}
f_1(x)=\begin{cases}f(x)\left(\frac{x}{g(x)}\right)^\alpha, \;x\neq0, \\
 \frac{f(0)}{(g'(0))^\alpha},\quad\quad\; x=0. \end{cases}
\end{equation}
Two decoupled ODEs for $q$ and $h$ are obtained from \eqref{gpeq} 
{\color{black}
separately} by the superposition principle according to the singularity:
\begin{eqnarray}
  q'(x)+iw g'(x)q(x)+\alpha g'(x)\frac{q(x)-c_0(1-e^{-iwg(x)})}{g (x)} &=& f_1(x), \label{gqeq}\\
  h'(x)+iw g'(x)h(x) +\alpha c_0g'(x)\frac{1-e^{-iwg(x)}}{g^{1-\alpha}(x)}&=& 0, \label{gheq}
\end{eqnarray}
where $c_0$ is an unknown parameter to be determined. Note that a minor trick was used in the splitting procedure by adding and then subtracting the term $\alpha c_0g'(x)\frac{1-e^{-iwg(x)}}{g^{1-\alpha}(x)}$. This minor modification  makes the Levin method effective in computing singular and oscillatory integrals.

{\color{black}
Letting $q_1(x)=\frac{q(x)-c_0(1-e^{-iwg(x)})}{g(x)}$,} the equation \eqref{gqeq} is simplified in a  clear form,
\begin{equation}\label{gq1eq}
    iwg'(x)c_0+g(x)q_1'(x)+[1+\alpha+iwg(x)]g'(x)q_1(x)=f_1(x).
\end{equation}
Note that the solution $q$ in \eqref{gqeq} might be oscillatory, while the new defined function $q_1$ is non-oscillatory. In fact, we rigorously prove the non-oscillation property of the solution of \eqref{gq1eq}  in the following lemma. To avoid distraction from the narrative of the new Levin method, its proof is given in Appendix A.
\begin{lemma}\label{lemma3}
    Suppose that $f_1\in C^{2n+1}[0,a]$ and $g\in C^{2n+2}[0,a]$ with $g(0)=0$ and $g'(x)>0, x\in[0,a]$. If $f_1$ and $g$ are independent of $w$, then there exist a function $q_1$ and a number $c_0$ satisfying \eqref{gq1eq} such that
    \begin{equation}\label{property}
    |c_0|<C/w \;\text{and}\; \|\bD^{j}q_1\|_\infty < C/w,\;j=0,1,\ldots,n,
    \end{equation}
    where $C$ is a constant independent of $w$.
 \end{lemma}

 Lemma \ref{lemma3} is the cornerstone of the proposed new method since 
 {\color{black}
 it ensures that the ODE \eqref{gq1eq} can be solved efficiently by the collocation method based on polynomials no matter how large the absolute value of $w$ is. }

Once the value of $c_0$ is known, a particular solution $h$ of \eqref{gheq} subject to the initial condition $h(0)=0$ is well-known by the standard ODE theory, given explicitly by
\begin{equation}\label{hsol}
\begin{split}
h(x) &=\alpha c_0e^{-iwg(x)}\int_0^x\frac{g'(t)(1-e^{iwg(t)})}{g^{1-\alpha}(t)}dt \\
&= \alpha c_0e^{-iwg(x)} \int_0^{g(x)}\frac{1-e^{iwt}}{t^{1-\alpha}}dt \\
&=c_0e^{-iwg(x)}\left(g^\alpha(x)+\frac{\alpha \Gamma(\alpha,-iwg(x)-\Gamma(\alpha+1)}{(-iw)^\alpha}\right),
\end{split}
\end{equation}
 where $\Gamma(s,z)$  is the incomplete gamma function \cite{STEGUN}. It is the reason why we choose $g^\alpha(x)$ to express the algebraic singularity. Instead, if $x^\alpha$ is adopted, it is difficult to evaluate the solution $h$ explicitly or numerically. 

{\color{black}
We now formally propose} the new Levin method for the integral $I_w^{[0,a]}[f,s_1,g]$. To this end, we define a new operator for a given function $g$, a number $w$ and a number $\alpha$:
\[
  \bW_{w,\alpha,g}[c_0,q_1]:=iwg'(x)c_0+g(x)q_1'(x)+[1+\alpha+iwg(x)]g'(x)q_1(x).
\]
For notational simplicity, the explicit dependence of the new operator on $w, \alpha$, and $g$ will be suppressed, to be understood only implicitly.  Let $\{\phi_j\}_{j=1}^{n}$ be a basis of functions independent of $w$. Moreover, let $\{x_j\}_{j=0}^{n}$ be a set of collocation nodes such that $0=x_0<x_1<\ldots<x_{n}=a$. We are seeking a pair of a function $q_1=\sum_{j=1}^{n} c_j\phi_j$ and a number $c_0$ such that they satisfy  \eqref{gq1eq} at the collocation nodes: this reduces to the linear system
\begin{equation}\label{interp1}
\bW[c_0,q_1](x_j)=f_1(x_j), j=0,1,\ldots,n,
\end{equation}
where $f_1$ is defined in \eqref{f1}. 
Written in the form of a vector, the system \eqref{interp1} becomes 
\[(\BA+iw\BB)\Bc=\Bf_1,\]
where the $(n+1)\times(n+1)$ matrices $\BA$ and $\BB$ are independent of $w$. Specifically, 
\[
  \BB=\begin{pmatrix}
  g'(0)   & 0 & 0 & \ldots & 0\\
  g'(x_1) & g(x_1)g'(x_1)\phi_1(x_1) & g(x_1)g'(x_1)\phi_2(x_1) &\ldots  & g(x_1)g'(x_1)\phi_n(x_1)  \\
  g'(x_2) & g(x_2)g'(x_2)\phi_1(x_2) & g(x_2)g'(x_2)\phi_2(x_2) &\ldots   & g(x_2)g'(x_2)\phi_n(x_2)  \\
  \vdots  & \vdots                   & \vdots                   &          \\
  g'(x_n) & g(x_n)g'(x_n)\phi_1(x_n) & g(x_n)g'(x_n)\phi_2(x_n) & \ldots & g(x_n)g'(x_n)\phi_n(x_n)
  \end{pmatrix},
\]
and hence the matrix $\BB$ is non-singular once the basis $\{\phi_j\}_1^{n}$ is a Chebyshev set \cite{OLVER2006}.

\begin{proposition}
For sufficiently large $w$, the system \eqref{interp1} has a unique solution. Moreover, its solution $q_1$ is slowly oscillatory and both $q_1$ and $c_0$ are $\bO(w^{-1})$ as $w\rightarrow \infty$.
\end{proposition}
\begin{proof}
The proof is trivial with the use of Cramer's rule following  \cite{OLVER2006} and  \cite{XIANGBIT2007}.
\end{proof}

With a little effort, this collocation is readily generalized, including confluent collocation nodes. Assuming that each collocation node $x_j$ is 
{\color{black}
accompanied by multiplicity $m_j\geq 1$} such that $\sum_{j=0}^{n}m_j-1=M$, we require not only the equivalence of the values of $f_1$ 
and $\bW[c_0,q_1]$ at the collocation nodes, but also the values of the derivatives of $f_1$ 
and $\bW[c_0,q_1]$, up to the given multiplicity. In place of $q_1=\sum_{j=1}^{n} c_j\phi_j$ and \eqref{interp1}, we pursue a function $q_1=\sum_{j=1}^{M} c_j\phi_j$ 
{\color{black}
that satisfies the linear system}
\begin{equation}\label{interp2}
\begin{split} 
\frac{d^{j}\bW[c_0,q_1]}{dx^j}(x_l)=f_1^{(j)}(x_l), j=0,1,\ldots,m_l-1, l=0,1,\ldots,n.
\end{split}
\end{equation}
Note that the preceding system \eqref{interp1} is a special case with $M=n$ and all multiplicities equal 1. 

Having obtained the solution of \eqref{interp2}  together with the formula \eqref{hsol}, we define the new Levin method
\begin{equation}\label{levin1}
 Q_{w,\alpha,n}^{L,s}[f]\equiv \left.\left[g^{\alpha+1}(x)q_1(x)+c_0(1-e^{-iwg(x)}) g^\alpha(x)+h(x)\right]e^{iwg(x)} \right|_0^a,
\end{equation}
where $s=\min(m_0,m_n)-1$.

In the following, we  consider the asymptotic order of the new Levin method. To this end, we recall a lemma concluded from the results in \cite{XAINGGUO2014} and \cite{GAO2016}.
\begin{lemma}\label{xlemma}
Suppose $f\in C^{s+1}[0,a]$,  $f^{(j)}(0)=f^{(j)}(a)=0, j=0,1,\ldots,s$ and every function in the set $\{f,f',\ldots,f^{(s+1)}\}$ is of asymptotic order $\bO(1), w\rightarrow \infty$, then 
\begin{eqnarray*}
  \int_0^a x^\alpha f(x) e^{iwx}dx  &\sim & \bO\left(w^{-s-1-\min\{1+\alpha,1\}}\right), \\
 \int_0^a x^\alpha \ln (x) f(x) e^{iwx}dx &\sim & \bO\left(\delta_\alpha(w) w^{-s-1-\min\{1+\alpha,1\}} \right),
\end{eqnarray*}
where
\begin{equation}\label{delta1}
\delta_{\alpha}(w):=\begin{cases} 1+|\ln(w)|, -1<\alpha\leq 0,\\ 1, \alpha>0.\end{cases}
\end{equation}
\end{lemma}

 \begin{theorem}\label{theorem3}
Suppose that  $g(0)=0$ and $g'(x)>0, x\in[0,a]$ and $m_0=m_n=s+1$. If the basis $\{\phi_j\}_{j=1}^{M}$ is a Chebyshev set where $M=\sum_{j=0}^{n}m_j-1$, then for sufficiently large $w$ the system \eqref{interp2} has a unique solution and
\begin{equation}
I_w^{[0,a]}[f,s_1,g]-Q_{w,\alpha,n}^{L,s}[f]\sim \bO(w^{-s-1-\min\{1+\alpha,1\}}).
\end{equation}
\end{theorem}
\begin{proof}
It is known from the fundamental theorem of calculus that
\[\begin{split}
  Q_{w,\alpha,n}^{L,s}[f] & =\int_0^a \bL\left[g^{\alpha+1}q_1+c_0(1-e^{-iwg})g^\alpha+h\right](x) e^{iwg(x)} dx \\
  &=\int_0^a \bW[c_0,q_1](x)g^\alpha(x)e^{iwg(x)}dx,
  \end{split}
\]
where the expression \eqref{gheq} for $\bL[h]$ has been used in the computation. It  follows that 
\[
I_w^{[0,a]}[f,s_1,g]-Q_{w,\alpha,n}^{L,s}[f]=\int_0^a \left( \bW[c_0,q_1](x)-f_1(x)\right)g^\alpha(x)e^{iwg(x)}dx,
\]
on the face of which we are almost done using Lemma \ref{xlemma}. 

To bridge the final gap, we only must show that $\bW[c_0,q_1]^{(j)}$ is $\bO(1), j=0,1,\ldots,s+1$, which is similar to the proof of Theorem 4.1 in \cite{OLVER2006} or of Theorem 3.5 in \cite{ISERLES2018B}. 

{\color{black}
Note that the linear system \eqref{interp2} can be written in the vector form $(\BA+iw\BB)\Bc=\Bf_1$}, where $\BA$ and $\BB$ are independent of $w$. 
{\color{black}
For sufficiently large $w$, $\det \BB \neq 0$ is sufficient to show} the unique existence of \eqref{interp2} and the boundedness of  $\bW[c_0,q_1]^{(j)}$. 
{\color{black}
Hence, all we must show is that the matrix $\BB$ is non-singular.}  Since the proof is identical in concept, we just prove for the case with $n=1$, $m_0=m_1=2$ and $x_0=0, x_1=a$. In this case, 
\[
  \BB=\begin{pmatrix}
  g'(0)   & 0 & 0 & 0\\
  g''(0) & \eta_1'(0) & \eta_2'(0) & \eta_3'(0) \\
  g'(a) & \eta_1(a) & \eta_2(a) & \eta_3(a)  \\    
  g''(a) & \eta_1'(a) & \eta_2'(a) & \eta_3'(a) 
  \end{pmatrix},
\]
where $\eta_k(x)=g(x)g'(x)\phi_k(x), k=1,2,3$. 
{\color{black}
Performing several row operations, we derive }
\[
  \det\BB=(g'(0))^3(g(a)g'(a))^2 \begin{vmatrix}
  \phi_1(0) & \phi_2(0) & \phi_3(0)\\
  \phi_1(a) & \phi_2(a) & \phi_3(a)\\
  \phi_1'(a) & \phi_2'(a) & \phi_3'(a)
  \end{vmatrix}.
\]
{\color{black}
The assumption of a Chebyshev set $\{\phi_k\}_{k=1}^3$ assures that the matrix on the right is non-singular. In addition to $g'(x)\neq 0$ and $g(a)\neq 0$,} it follows that $\det\BB\neq 0$. The proof is finished.
\end{proof}

{\color{black}
In addition to the asymptotic order, the precision of the Levin method also relies on the number of collocation nodes.} To show the dependence, we consider a special case of linear oscillator.
 {\color{black}
 We also set the multiplicity of each point to 1.}
Let $E_n(f)=\left|I_w^{[0,a]}[f,s_1,\tau]-Q_{w,\alpha,n}^{L,0}[f]\right|$ denote the absolute error where $\tau(x):=x$.
\begin{theorem}\label{theom1}
If $f$ is suitably smooth and independent of $w$, then the new Levin method collocating on points $\{0= x_0<x_1<\ldots<x_{n}\leq a\}$ satisfies
\begin{equation}
    E_n(f)\leq Cw^{-\min\{1+\alpha,1\}} \frac{\|f^{(n+1)}\|_\infty a^{n+1} }{n!},
\end{equation}
where $C$ is a constant independent of $w$ and $n$.
\end{theorem}
\begin{proof}
It is already known from preceding analysis that
\begin{equation}
    E_n(f)=\left|\int_0^a (f(x)-\bW[c_0,q_1](x))x^\alpha e^{iwx}dx\right|=\left|\int_0^a (f(x)-p(x))x^\alpha e^{iwx}dx\right|,
\end{equation}
where $p$ is the interpolation of $f$ on the nodes $0= x_0<x_1<\ldots<x_{n}\leq a$. 
{\color{black}
To estimate the error,}
let $\eta(x):=f(x)-p(x)$. It is obvious that $\eta(x_j)=0, j=0,1,\ldots,n$. According to Rolle's theorem, there exists $y_j\in (x_j, x_{j+1})$ such that
\[
\eta'(y_j)=0,\; j=0,1,\ldots,n-1.
\]
Using the expression for interpolation errors, we derive
\[
\eta(x)=\frac{\eta^{(n+1)}(\xi_1)}{(n+1)!}\prod_{j=0}^{n} (x-x_j), \;\; \eta'(x)=\frac{\eta^{(n+1)}(\xi_1)}{n!}\prod_{j=0}^{n-1} (x-y_j),
\]
where $\xi_1,\xi_2\in[0,a]$ depending on the value of $x$.
By the van der Corput-type lemma in \cite{XAINGGUO2014}, there exists a constant $C$ independent of $w$ and $n$ such that
\begin{equation}\label{eq15}
\begin{split}
E_n(f)&\leq Cw^{-\min\{1+\alpha,1\}}\left(|\eta(a)|+ \int_0^a \left|\eta'(x)\right|dx\right)\\
&\leq Cw^{-\min\{1+\alpha,1\}}\left(\|\eta\|_\infty+ a\left\|\eta'\right\|_\infty \right).
\end{split}
\end{equation}
{\color{black}
The desired inequality follows directly from the fact} that $p^{(n+1)}\equiv 0$ and $\eta^{(n+1)}=f^{(n+1)}$.
\end{proof}

Note that the dependence on the number of nodes is related to the interpolation errors of the function $f$ and its derivative. Once the function $f$ is analytic within an ellipse and collocation points are chosen to be the Chebyshev points, the proposed method possesses the superalgebraic convergence 
{\color{black}
since the corresponding interpolation errors decrease supralgebraically} \cite{XIANG2010chebyshev}. 
{\color{black}
Thus, the new Levin method requires a small number of nodes to attain machine precision} that is also uniformly efficient for small $w$.

In computation of highly oscillatory integrals without singularity, it is known that the Filon-type method is equivalent to the Levin method once we use a proper basis \cite{OLVER2006,ISERLES2018B,XIANG2007b}. This conclusion is readily generalized to the highly oscillatory integrals with algebraic singularity.
Assuming that $g'\neq 0$ in $[0,a]$, we define two sets
\[
  \Psi_M=\{g',g'g,g'g^2,\ldots,g'g^{M-1}\}
\]
and
\[
  \Phi_M=\{1,g,g^2,\ldots,g^{M-1}\}.
\]
Note that when $g$ is strictly monotone, $\Psi_N$ and $\Phi_N$ are both Chebyshev sets. Let $\varphi_k=g'g^{k-1}, k=1,2,\ldots,$. Suppose that $p(x)=\sum_{j=1}^{M+1}p_j\varphi_j$ where $M=\sum_{j=0}^nm_j-1$ is the solution to the linear system
\begin{equation}\label{interpf}
p^{(j)}(x_l)=f_1^{(j)}(x_l), j=0,1,\ldots,m_l-1, l=0,1,\ldots,n,
\end{equation}
where $f_1$ is defined in \eqref{f1}.
A new Filon-type method for $I_w^{[0,a]}[f,s_1,g]$ is defined by
\begin{equation}\label{filon1}
Q_{w,\alpha,n}^{F,s}[f]\equiv\int_0^ap(x)g^\alpha(x)e^{iwg(x)}dx=\sum_{j=1}^{M+1}p_j\mu_j,
\end{equation}
where $\mu_j$ are the generalized moments defined by
\[
  \mu_j=\int_0^a\varphi_j(x)g^\alpha(x)e^{iwg(x)}dx.
\]
They can be evaluated fast by a recurrence relation,
\[
  \mu_{j+1}=-\frac{j+\alpha}{iw}\mu_j+\frac{1}{iw} g^{j+\alpha}(a)e^{iwg(a)}, j=1,2,\ldots,
\]
and 
\[
  \mu_1=\frac{g^{1+\alpha}(a)}{(-iwg(a))^{1+\alpha}}\left[\Gamma(1+\alpha)-\Gamma(1+\alpha,-iwg(a))\right].
\]

\begin{theorem}\label{theorem2}
The Filon-type method \eqref{filon1} based on the basis set $\Psi_{M+1}$ is identical to the Levin method \eqref{levin1} using the basis set $\Phi_M$.
\end{theorem}
\begin{proof}
It is trivial that the interpolant $p$ in the Filon-type method and the function $\bW[c_0,q_1]$ in the Levin method are in the same space:  they both belong to $\text{span}\{\Psi_{M+1}\}$. 
Moreover, both $p$ and $\bW[c_0,q_1]$ obey the Hermite interpolation conditions \eqref{interp2}.  According to the uniqueness of the Hermite interpolation, we derive that $p=\bW[c_0,q_1]$. The equivalence of these two methods follows directly. 
\end{proof}

\begin{remark}
As Olver pointed out in \cite{OLVER2010,OLVER2010b}, 
{\color{black}
how to construct the Filon-type method in a numerically stable manner with the basis set $\Psi_{M+1}$ is still unknown, as is how to choose the interpolation points to optimize the order of convergence.} However, the new Levin method can be implemented by the polynomial interpolation at Chebyshev points, which can ensure the convergence. Compared to the Filon-type method, the Levin method is more stable and accurate for the case of nonlinear oscillators. This is a merit of the Levin method.
\end{remark}

\section{New Levin method for \texorpdfstring{$I_w^{[0,a]}[f,s_2,g]$}{I2}}

{\color{black}
We now  further consider a new Levin method} for the case of oscillatory integrals with both algebraic and logarithmic singularities.

Before we commence the development of the new Levin method, it behooves us to decompose the integral $I_w^{[0,a]}[f,s_2,g]$:
\begin{equation}\label{eq1}
\begin{split}
I_w^{[0,a]}[f,s_2,g] & =\int_0^a f_1(x)g^{\alpha}(x)\log g(x)e^{iwg(x)}dx+\int_0^a f_2(x)x^{\alpha}e^{iwg(x)}dx
 \\
&\triangleq I_1(f_1)+I_2(f_2),
\end{split}
\end{equation}
where $f_1$ is defined in \eqref{f1}
and
\begin{equation}\label{f2}
f_2(x)=\begin{cases}f(x)  \log \frac{x}{g(x)}, \;x\neq0, \\
f(0)\log \frac{1}{g'(0)}, \;x=0.\end{cases}
\end{equation}
{\color{black}
It is obvious that $I_2(f_2)=I_w^{[0,a]}[f_2,s_1,g]$, which is readily computed by $Q_{w,\alpha,n}^{L,s}[f]$. All we need to do is to evaluate the integral $I_1(f_1)$.}

To compute $I_1(f_1)$, the spirit of the classic Levin method requires the solution of the ODE:
\begin{equation}\label{levinode2}
p'(x)+iwg'(x)p(x)=f_1(x)g^\alpha(x)\log(g(x)), x\in(0,a].
\end{equation}
{\color{black}
It is not wise to solve the ODE directly due to the singularity on the right-hand side. To deal with this obstacle, we combine the techniques described in the preceding section and in \cite{WangXiang2018} to seek a particular solution of a form with its singularity explicitly represented: }
\[
p(x)=q(x)g^{\alpha}(x) \log g(x)+\ell(x)g^{\alpha}(x)+h(x),
\]
where $q,\ell$ and $h$ are unknown functions. Substituting the form of $p$ in \eqref{levinode2}, we derive
\begin{equation*}
\begin{split}
 & \left(q'(x)+iwg'(x)q(x)+\alpha g'(x)\frac{q(x)}{g(x)}\right) g^{\alpha}(x)\log g(x) \\
&\quad\quad\quad  +\left(\ell'(x)+iwg'(x)\ell(x)+\alpha g'(x)\frac{\ell(x)}{g(x)}+g'(x)\frac{q(x)}{g(x)} \right) g^{\alpha}(x) \\
&\quad\quad\quad\quad\quad\quad  +h'(x)+iwg'(x)h(x)=f_1(x)g^{\alpha}(x)\log g(x).
\end{split}
\end{equation*}
By the superposition principle, we then split the above ODE under the criterion of singularity:
\begin{equation}\label{gqeq2}
 q'(x)+iwg'(x)q(x)+\alpha g'(x)\frac{q(x)-c_0(1-e^{-iwg(x)})}{g(x)}=f_1(x),
\end{equation}
\begin{equation}\label{gleq2}
\begin{split}
& \ell'(x)+iwg'(x)\ell(x)+\alpha g'(x)\frac{\ell(x)-c_1(1-e^{-iwg(x)})}{g(x)}
+g'(x)\frac{q(x)-c_0(1-e^{-iwg(x)})}{g(x)}=0,
 \end{split}
\end{equation}
\begin{equation}\label{gheq2}
\begin{split}
   & h'(x)+iwg'(x)h(x)+\alpha c_0 g'(x)\frac{1-e^{-iwg(x)}}{(g(x))^{1-\alpha}}\log g(x)
   \\
    & \quad\quad\quad\quad\quad
    +\alpha c_1 g'(x)\frac{1-e^{-iwg(x)}}{(g(x))^{1-\alpha}}+c_0g'(x) \frac{1-e^{-iwg(x)}}{(g(x))^{1-\alpha}}=0.
    \end{split}
\end{equation}

Setting
\[
\begin{split}
q_1(x)=\frac{q(x)-c_0(1-e^{-iwg(x)})}{g(x)}
\;\;\text{and}\; \ell_1(x)=\frac{\ell(x)-c_1(1-e^{-iwg(x)})}{g(x)},
\end{split}
\]
{\color{black}
Eqs. \eqref{gqeq2} and \eqref{gleq2} are simplified as}
\begin{eqnarray}
  iwg'(x)c_0+g(x)q_1'(x)+[1+\alpha+iwg(x)]g'(x)q_1(x) &=& f_1(x), \label{gq1eq2}\\
  iwg'(x)d_0+g(x)\ell_1'(x)+[1+\alpha+iwg(x)]g'(x)\ell_1(x) &=& -q_1(x)g'(x) \label{gl1eq2},
\end{eqnarray}
{\color{black}
which have exactly the same form of the ODE \eqref{gq1eq}. This means that both \eqref{gq1eq2} and \eqref{gl1eq2} possess at least one well-regularized and non-oscillatory solution that can be solved efficiently by collocation methods. Regarding Eq. \eqref{gheq2},} there is an explicit solution subject to the initial condition $h(0)=0$:
\begin{equation}\label{ghsol22}
\begin{split}
    h(x)&=e^{-iwg(x)}\left(\int_0^x \alpha c_0 g'(t)\frac{1-e^{iwg(t)}}{g^{1-\alpha}(t)}\log g(t) +\alpha c_1 g'(t)\frac{1-e^{iwg(t)}}{g^{1-\alpha}(t)}+c_0g'(t) \frac{1-e^{iwg(t)}}{g^{1-\alpha}(t)}dt\right)\\
    &=\left( c_0\log g(x) + c_1+\frac{c_0}{\alpha}\right)  \left(g^{\alpha}(x)+\frac{\alpha \Gamma(\alpha,-iwg(x))-\Gamma(\alpha+1)} {(-iw)^\alpha} \right)e^{-iwg(x)} \\
    &\quad\quad\quad\quad\quad\quad\quad\quad\quad\quad +\frac{c_0}{\alpha}g^{\alpha}(x)(_2F_2(\alpha,\alpha;1+\alpha,1+\alpha;iwg(x))-1)e^{-iwg(x)},
    \end{split}
\end{equation}
where $_mF_n$ is the generalized hypergeometric function, defined as a power series,
\begin{equation}\label{fdef}
_mF_n(a_1,\ldots,a_m;b_1,\ldots,b_n;z):=\sum_{k=0}^{+\infty} \frac{(a_1)_k\ldots(a_m)_k}{(b_1)_k\ldots(b_n)_k} \frac{z^k}{k!},
\end{equation}
and 
{\color{black}
$(b)_k$ is known as a Pochhammer symbol,} i.e., $(b)_0=1$ and $(b)_k=\frac{\Gamma(k+b)}{\Gamma(b)}$ for $k\geq 1$. 

Let $\{\phi_j\}_{j=1}^{M}$ be a basis of functions independent of $w$ and $\{x_j\}_{j=0}^{n}$ be a set of collocation nodes accompanied with multiplicity $m_j\geq 1$ such that $0=x_0<x_1<\ldots<x_{n}=a$ and $\sum_{j=0}^{n}m_j-1=M$. We are seeking two functions, $q_1=\sum_{j=1}^{M} c_j\phi_j$ and $\ell_1= \sum_{j=1}^{M} d_j\phi_j$, and two numbers, $c_0$ and $d_0$ such that they obey two linear systems, respectively,
\begin{equation}\label{interp3}
\begin{split} 
\frac{d^{j}\bW[c_0,q_1]}{dx^j}(x_l) & =f_1^{(j)}(x_l),\quad\qquad j=0,1,\ldots,m_l-1, l=0,1,\ldots,n. \\
\frac{d^{j}\bW[d_0,\ell_1]}{dx^j}(x_l) & =(-q_1g')^{(j)}(x_l), j=0,1,\ldots,m_l-1, l=0,1,\ldots,n. 
\end{split}
\end{equation}
We define the new Levin method for $I_w^{[0,a]}[f,s_2,g]$ as 
\begin{equation}\label{levin2}
\begin{split}
Q_{w,n}^{L,s}[f]\equiv Q_{w,\alpha,n}^{L,s}[f_2] & +\left[(q_1(x)\log(g(x))+\ell_1(x))g^{1+\alpha}(x)\right. \\ & \left. \left.-(c_0\log(g(x))+d_0)(1-e^{-iwg(x)})g^\alpha(x)+h(x)\right]e^{iwg(x)}\right|_0^a,
\end{split}
\end{equation}
where $h$ is given in \eqref{ghsol22} and $f_2$ is defined in \eqref{f2}.

We next show the asymptotic order of the new Levin method.
\begin{theorem}
Suppose that  $g(0)=0$ and $g'(x)>0, x\in[0,a]$ and $m_0=m_n=s+1$. If the basis $\{\phi_j\}_{j=1}^{M}$ is a Chebyshev set where $M=\sum_{j=0}^{n}m_j-1$, then for sufficiently large $w$ each of the two systems \eqref{interp3} has a unique solution and
\begin{equation}
I_w^{[0,a]}[f,s_2,g]-Q_{w,n}^{L,s}[f]\sim \bO(\delta_\alpha(w)w^{-s-1-\min\{1+\alpha,1\}}).
\end{equation}
\end{theorem}
\begin{proof} Similar to the proof of Theorem \ref{theorem3}, 
{\color{black}
we commence with the representation of  $Q_{w,n}^{L,s}[f]$ in integral form.}
 Using the fundamental theorem of calculus, we derive
\[\begin{split}
  Q_{w,n}^{L,s}[f]= Q_{w,\alpha,n}^{L,s}[f_2] & + \int_0^a \bL\left[(q_1\log(g)+\ell_1)g^{1+\alpha}\right. \\   &  \left.-(c_0\log(g)+d_0)(1-e^{-iwg})g^\alpha +h \right](x) e^{iwg(x)} dx,\\
  =  Q_{w,\alpha,n}^{L,s}[f_2] & + \int_0^a \bW[c_0,q_1](x)g^\alpha(x)\log(g(x)) e^{iwg(x)} dx \\ 
  & + \int_0^a \left(\bW[d_0,\ell_1](x) +q_1(x)g'(x)\right) g^\alpha(x) e^{iwg(x)} dx.
  \end{split}
  \]
Hence 
\[
\begin{split}
I_w^{[0,a]}[f,s_2,g]-Q_{w,n}^{L,s}[f]= & I_w^{[0,a]}[f_2,s_1,g]-Q_{w,\alpha,n}^{L,s}[f_2] \\
& \quad+ \int_0^a \left(f_1(x) - \bW[c_0,q_1](x)\right)g^\alpha(x)\log(g(x))e^{iwg(x)}dx \\
& \qquad - \int_0^a \left(\bW[d_0,\ell_1](x) +q_1(x)g'(x)\right) g^\alpha(x) e^{iwg(x)} dx.
\end{split}
\]
Finally, we use Lemma 2.3 and Theorem 2.4 in a manner similar to the proof of Theorem 2.4 and the desired results follow.
\end{proof}

Similar to the case of algebraic singularity,  the error bound of the proposed Levin method on the number of nodes closely depends on the interpolation errors of the related
functions. We list the result without proof for a special case of linear oscillator. 
Let $E_n(f):=\left|I_w^{[0,a]}[f,s_2,\tau]-Q_{w,n}^{L,0}[f]\right|$ denote the absolute error where $\tau(x):=x$.
\begin{theorem}
If $f$ is suitably smooth and independent of $w$, then the new Levin method collocating on points $\{0\leq x_0<x_1<\ldots<x_{n}\leq a\}$ satisfies
\begin{equation}
    E_n(f)\leq C\delta_\alpha(w)w^{-\min\{1+\alpha,1\}} \frac{\|f^{(n+1)}\|_\infty a^{n+1} }{n!},
\end{equation}
where $C$ is a constant independent of $w$ and $n$.
\end{theorem}

Inspired by the new Levin method,  a new moment-free Filon-type method is readily developed for $I_w^{[0,a]}[f,s_2,g]$. 
{\color{black}
We find a function $p(x)=\sum_{j=1}^{M+1}p_j\varphi_j$ where $M=\sum_{j=0}^nm_j-1$ and $\varphi_j=g'g^{j-1},j=1,\ldots,M+1$,} such that
\begin{equation}\label{interpf2}
p^{(j)}(x_l)=f_1^{(j)}(x_l), j=0,1,\ldots,m_l-1, l=0,1,\ldots,n,
\end{equation}
where $f_1$ is defined in \eqref{f1}.
We define a new Filon-type method for $I_w^{[0,a]}[f,s_2,g]$ by
\begin{equation}\label{filon2}
\begin{split}
Q_{w,n}^{F,s}[f] & \equiv  Q_{w,\alpha,n}^{F,s}[f_2]+\int_0^ap(x)g^\alpha(x)\log(g(x))e^{iwg(x)}dx \\
& = Q_{w,\alpha,n}^{F,s}[f_2]+ \sum_{j=1}^{M+1}p_j\nu_j,
\end{split}
\end{equation}
where $\nu_j$ are the generalized moments defined by
\[
  \nu_j=\int_0^a\varphi_j(x)g^\alpha(x)\log(g(x))e^{iwg(x)}dx.
\]
They can be evaluated fast by a recurrence relation,
\[
  \nu_{j+1}=-\frac{j+\alpha}{iw}\nu_j-\frac{1}{iw}\mu_j+\frac{1}{iw} g^{j+\alpha}(a)\log(g(a))e^{iwg(a)}, j=1,2,\ldots,
\]
and 
\[
  \nu_1=\frac{\log(g(a))}{(-iw)^{1+\alpha}}\left[\Gamma(1+\alpha)-\Gamma(1+\alpha,-iwg(a))\right] -\frac{g^{1+\alpha}(a)}{(1+\alpha)^2} {_2}F_2(1+\alpha,1+\alpha;2+\alpha,2+\alpha;iwg(a)).
\]

An identical reasoning of Theorem 2.6 reveals the relation between the Filon-type method and the Levin method.
\begin{theorem}
The Filon-type method \eqref{filon2} based on the basis set $\Psi_{M+1}$ are identical to the Levin method \eqref{levin1} using the basis set $\Phi_M$.
\end{theorem}

\section{Collocation method for \texorpdfstring{\eqref{gq1eq}}{(2.5)}}
{\color{black}
As seen in Sections 2 and 3,} new Levin methods depend on numerical solution of the kind ODE \eqref{gq1eq}. The ODE can be  solved  either in the frequency space (i.e., to obtain coefficients of the basis functions) or in the physical space (i.e., to obtain values of function in the collocation points). It is suggested from \eqref{levin1} and \eqref{levin2} that the solution solved in the physical space is more convenient to avoid the recovery process from the expanding expression. 
{\color{black}
Because many integrals of interest appeared in high-frequency scatterings,  only point values of $f$ can be used since $f$ is often very complicated (and may itself be an integral involving special functions);  we therefore propose in this section a new collocation method for \eqref{gq1eq} without any derivative information, i.e., $s=0$.We also note that derivatives might be avoided by allowing  interpolation points close to the critical points  as $w$ increases \cite{ISERLES2004b}.}

There exists a stable collocation method in the physical space to solve the classic Levin ODE for oscillatory integrals without singularities \cite{LI2008}. Unluckily, it is not applicable directly for  \eqref{gq1eq} since there is an extra unknown coefficient $c_0$ to be decided. 
{\color{black}
To circumvent this difficulty,}
 we adopt the Chebyshev-Gauss-Radau points, $t_j=-\cos \frac{2j\pi}{2n-1}, j=0,1,\ldots,n-1$, instead of Chebyshev-Lobatto nodes. We commence by recalling the first-order differentiation matrix $D$ based on the Chebyshev-Gauss-Radau points. 
 {\color{black}
 The matrix is determined by an explicit formula the entries of which are given by} (\cite{SHTW2011}, p. 100)
\begin{equation}\label{Dcgr}
    d_{kj}=
    \begin{cases}
    -\frac{n(n-1)}{3}, k=j=0,\\
    \frac{t_k}{2(1-t_k^2)}+\frac{(2n-1)T_{n-1}(t_k)}{2(1-t_k^2)Q'(t_k)}, \,\, 1\leq k=j\leq n-1, \\
    \frac{Q'(t_k)}{Q'(t_j)}\frac{1}{t_k-t_j}, k\neq j,
    \end{cases}
\end{equation}
where $Q(t)=T_n(t)+T_{n-1}(t), t\in[-1,1]$. 

{\color{black}
Letting $x(t)=\frac{a}{2}\left(1-t\right)$, $t\in[-1,1]$,}
 we select the modified Chebyshev-Gauss-Radau points  $x_j=x(t_{n-j}),j=1,\ldots,n$ and $x_0=0$ as the collocation points. 
 Then, for a given polynomial $u$ of degree less than $n$ on $[0,a]$, there exists the relation
\[
\Bu'=\frac{2}{a}D\Bu,
\]
where $\Bu$ and $\Bu'$ are two vectors of evaluations of functions $u$ and $u'$ at the modified Chebyshev-Gauss-Radau points, respectively.

{\color{black}
To derive a linear system for \eqref{gq1eq}, we must tackle the collocation condition at $x_0$ carefully. 
Since the coefficient of $q_1'(0)$ is $g(x_0)$ and $g(x_0)=0$, there is no need to express $q_1'(0)$ in terms of the values of $q$. 
Owing to the use of Chebyshev-Gauss-Radau points, the value of $q_1(0)$ must be represented by the extrapolation.} Using the interpolant of $q_1$, it is obtained by
\[
q_1(0)=\frac{1}{2n-1}\cos((n-1)\pi) q_1(x_n) +\sum_{j=1}^{n-1} \frac{2}{2n-1} \left(\cos((n-1-j)\pi)\sec \frac{j\pi}{2n-1}\right)  q_1(x_{n-j}),
\]
if $q_1$ is a polynomial of degree no more than $n-1$.

Let $\Bx:=[x_0,x_1,\ldots,x_n]$ and denote by $\Bq_1$ and $\Bf$ the modified vectors of values of functions $q_1$ and $f$, respectively,
i.e.,
\[
\Bq_1=[c_0, q_1(x_1), \ldots, q_1(x_n)],\; \text{and}\; \Bf=[f(x_0), f(x_1),\ldots,f(x_n)].
\]
{\color{black}
Set $\Br$ as a vector of the coefficients of $q_1(0)$ in terms of $q_1(x_j),j=1,2,\ldots,n$ 
 and $\Bc$ as a vector of size $n\times 1$ the entries of which equal $iw$.}
  We assemble a matrix $L$ of size $(n+1)\times (n+1)$ by
\[
L=\begin{pmatrix}iw & (1+\alpha)\Br^\top \\ \Bc & \Lambda_1D+\Lambda_2 \end{pmatrix},
\]
where $\Br^\top$ means the transpose of the vector $\Br$, $\Lambda_1=\diag\left(-\frac{2g(x_1)}{a},-\frac{2g(x_2)}{a},\ldots,-\frac{2g(x_n)}{a}\right)$ and $\Lambda_2=\diag(1+\alpha+iwg(x_1),1+\alpha+iwg(x_2),\ldots,1+\alpha+iwg(x_n)$.

Equation \eqref{gq1eq} is discretized on the collocation points $x_j(j=0,1,\ldots,n)$, and then we obtain the linear system in the vector form
\begin{equation}\label{Aeq2}
    L\Bq_1=\Bf.
\end{equation}
Note that the matrix $L$ is ill-conditioned when the dimension is large. However, as observed in \cite{LI2008}, only the last singular value of $L$ is very small and it is well separated from the rest. 
{\color{black}
Hence the technique of truncated singular value decomposition (TSVD)} is suggested to be used when the last singular value is smaller than $10^{-8}$ to obtain a stable solution.

\section{Numerical Examples}

In this section, we illustrate the convergence characteristics of proposed new Levin methods with a number of numerical experiments. 
We also compare the computational performance of the new collocation method with that of the composite moment-free Filon-type quadrature (CMFP) proposed in \cite{MAXU2017}. 
{\color{black}
The numerical results presented below were all obtained using MatLab (MathWorks, USA) on a laptop with an Intel(R) Core(TM) i7-6500U CPU with 8 GB of RAM.}

\begin{example}\label{example1}
 We first consider the integral with algebraic singularity considered in \cite{PIESSENS1992}, 
    \[
    \int_0^1 f(x)x^\alpha e^{-iwx}dx= \frac{1}{2(\alpha+1)}-\frac{\sqrt{\pi}}{2} \left(2/w\right)^{\alpha+1/2}\Gamma(\alpha+1)H_{\alpha+3/2}(w) +\frac{i\sqrt{\pi}}{2}\left(2/w\right)^{\alpha+1/2}\Gamma(\alpha+1) J_{\alpha+3/2}(w),
    \]
    where $f(x)=e^{iw}(1-x)(2-x)^\alpha$, $H_v$ is the Struve function and can be expressed in terms of the generalized hypergeometric function $_1F_2$,
    \[
    H_v(z)=\frac{z^{v+1}}{2^v\sqrt{\pi}\Gamma(v+3/2)}  {_1F_2}\left(1,3/2,v+3/2,-z^2/4 \right).
    \]    
\end{example}

\begin{figure}
\centering
\subfloat[$s=0$,$\alpha=0.5$]{
\label{fig1:subfig:a} 
\includegraphics[width=3.0in]{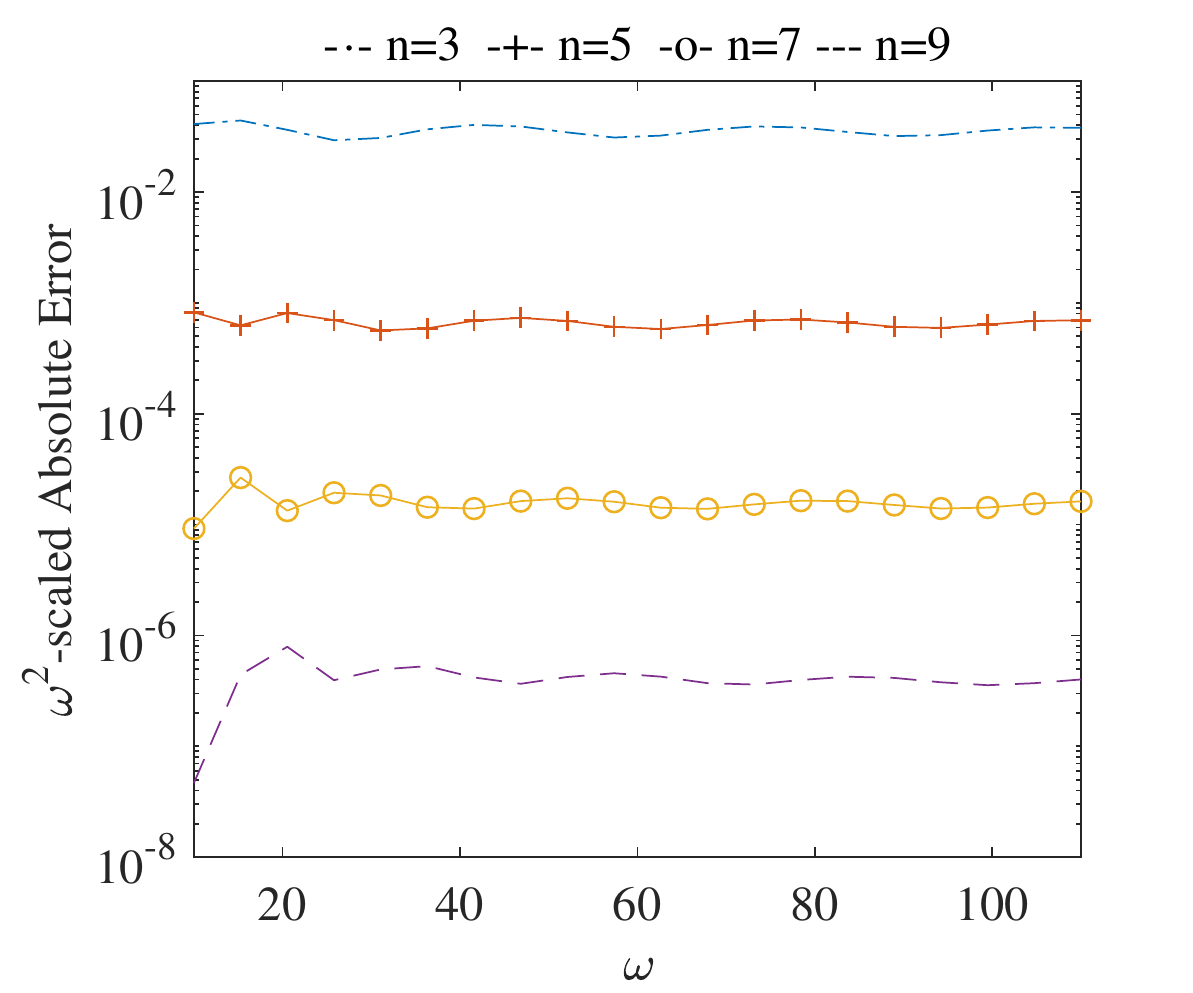}}
\hspace{0.1in}
\subfloat[$s=0$,$\alpha=-0.5$]{
\label{fig1:subfig:b} 
\includegraphics[width=3.0in]{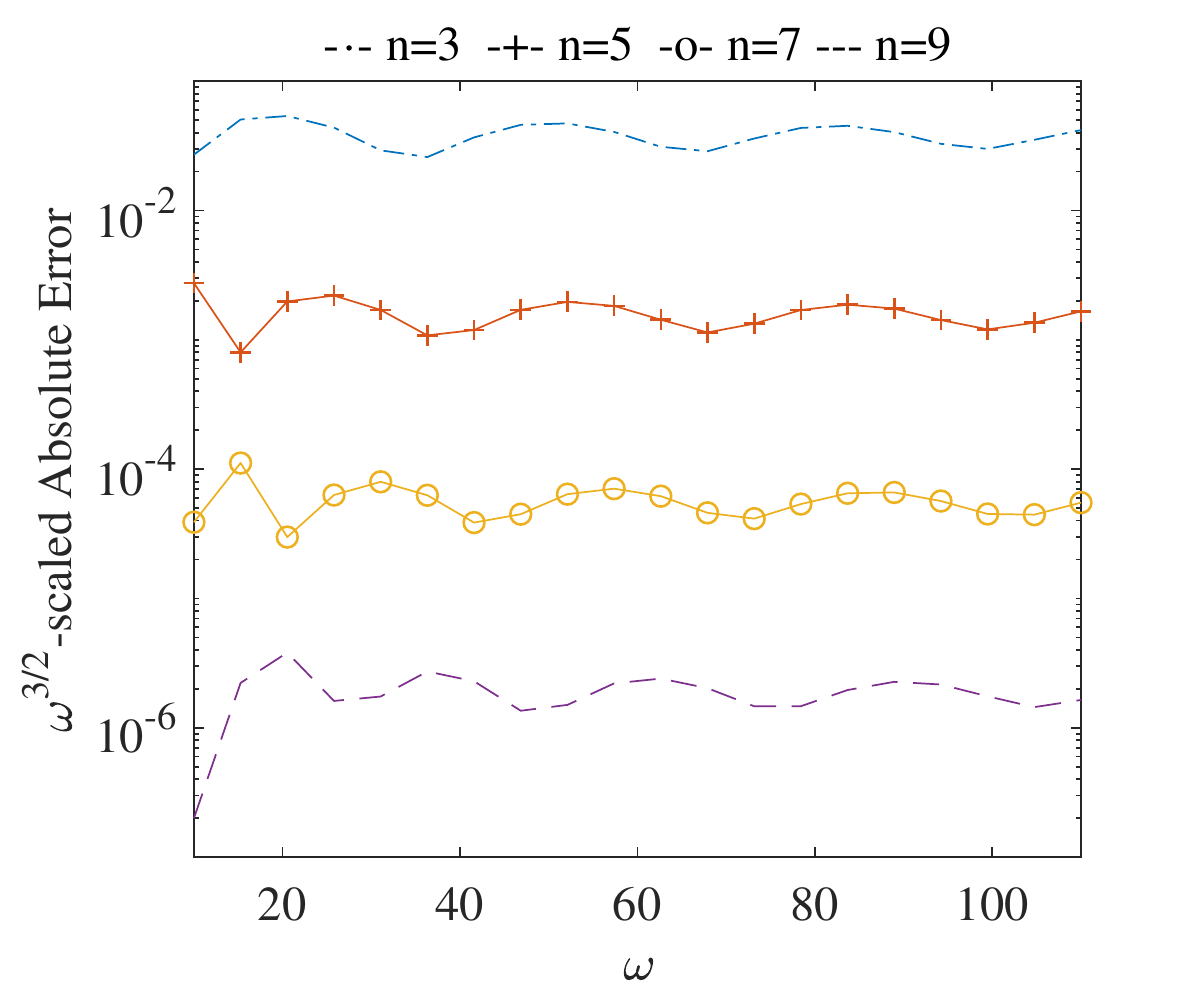}}

\subfloat[$s=1$,$\alpha=0.5$]{
\label{fig1:subfig:c} 
\includegraphics[width=3.0in]{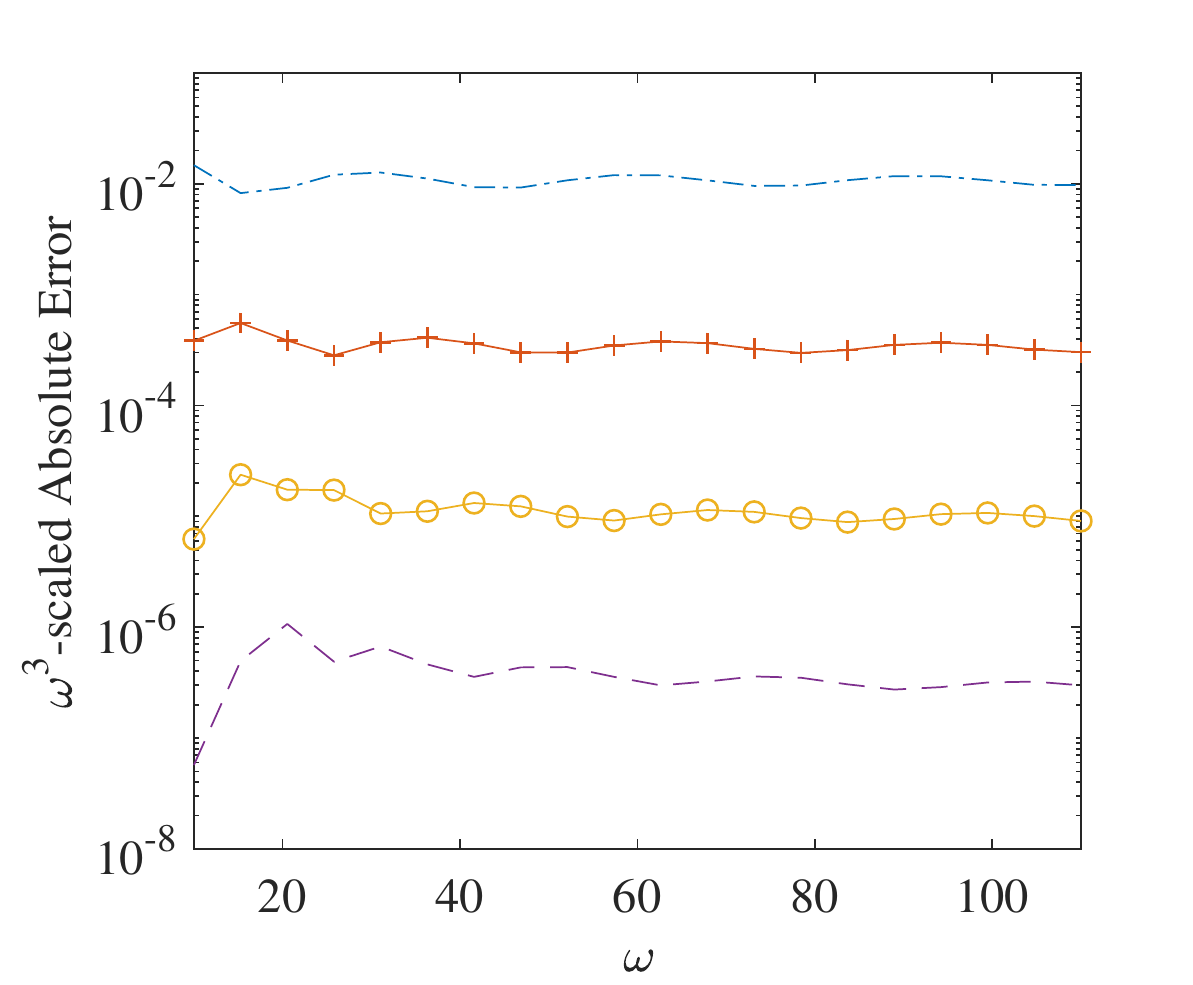}}
\hspace{0.1in}
\subfloat[$s=1$,$\alpha=-0.5$]{
\label{fig1:subfig:d} 
\includegraphics[width=3.0in]{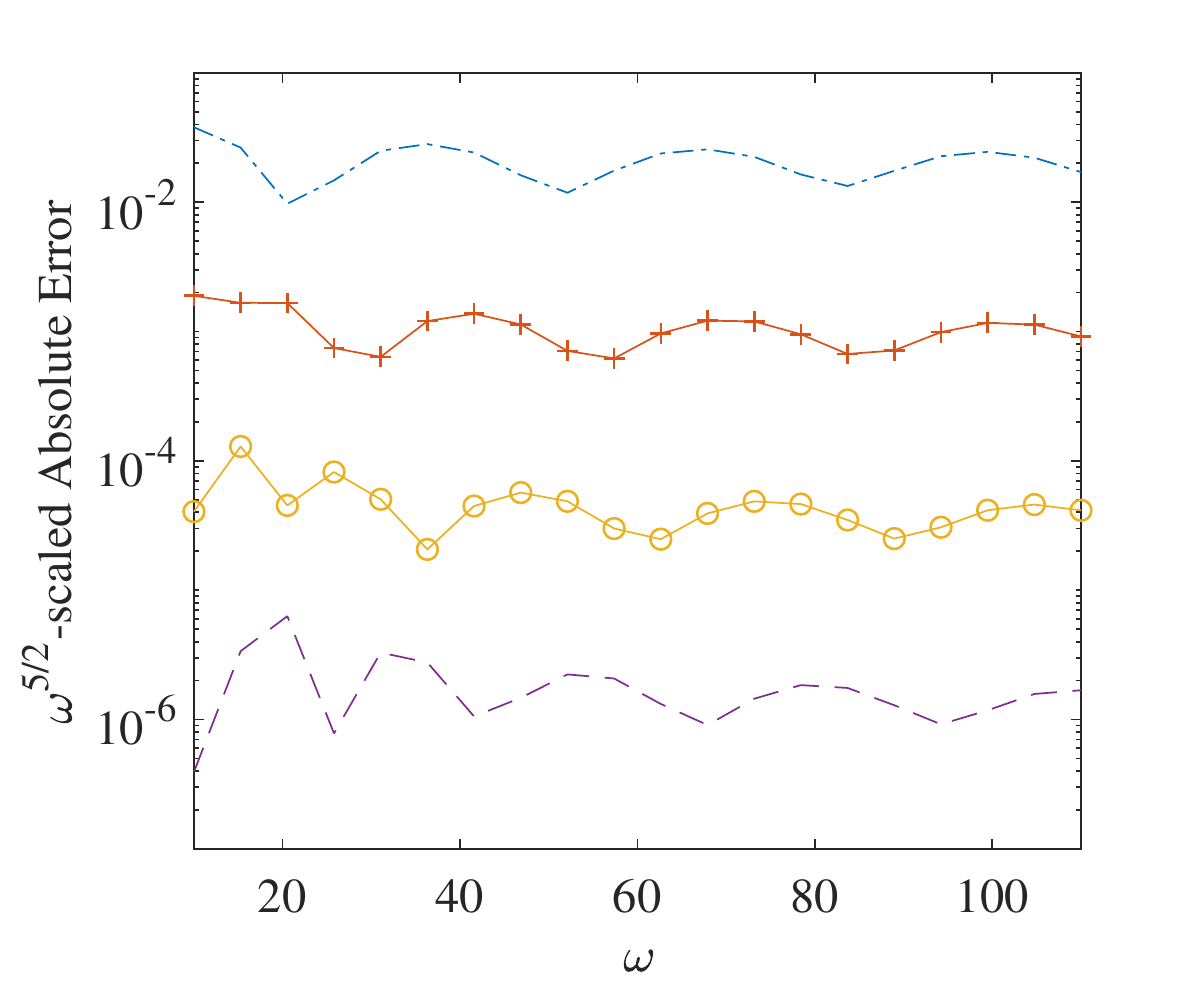}}

\subfloat[$s=2$,$\alpha=0.5$]{
\label{fig1:subfig:e} 
\includegraphics[width=3.0in]{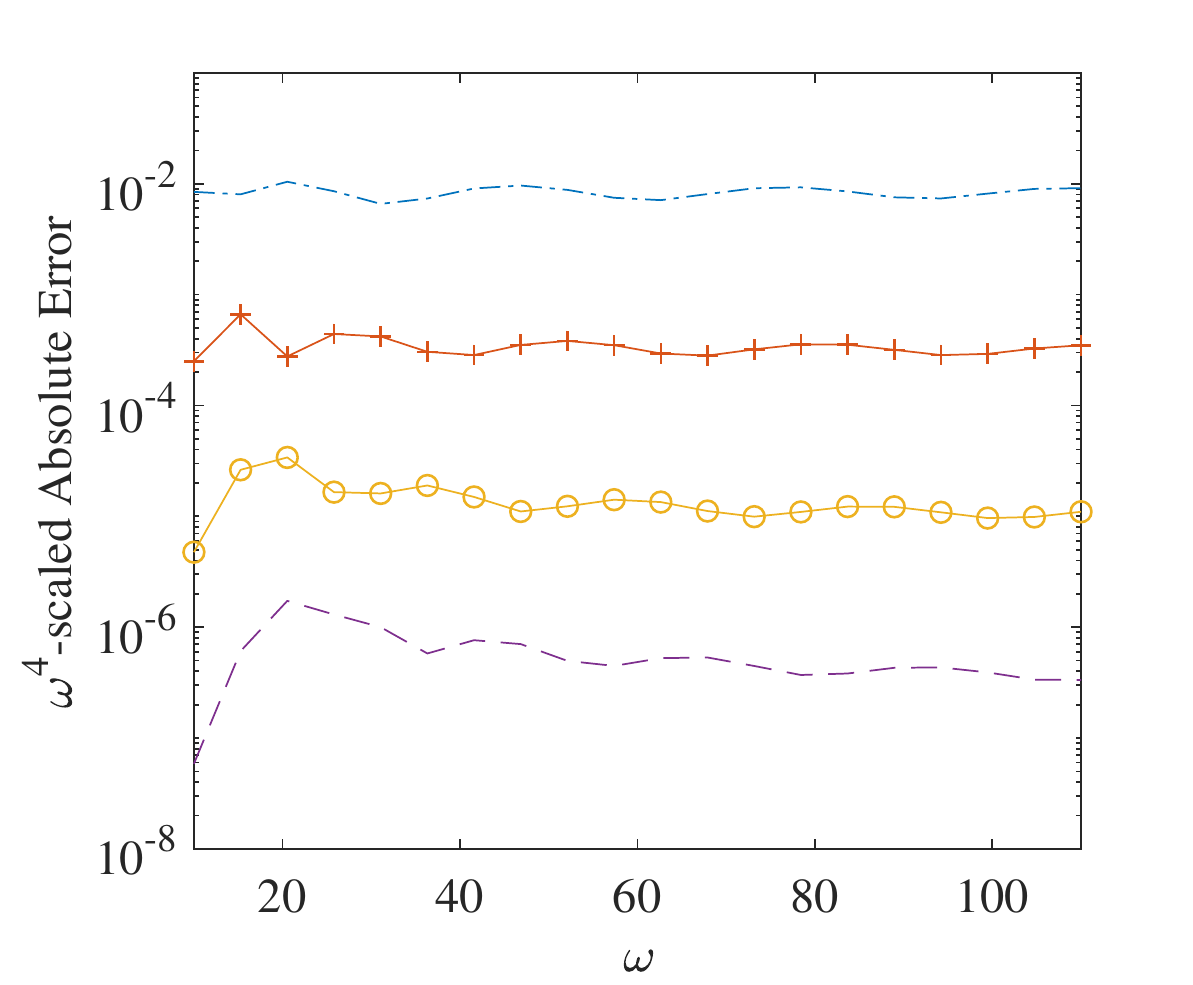}}
\hspace{0.1in}
\subfloat[$s=2$,$\alpha=-0.5$]{
\label{fig1:subfig:f} 
\includegraphics[width=3.0in]{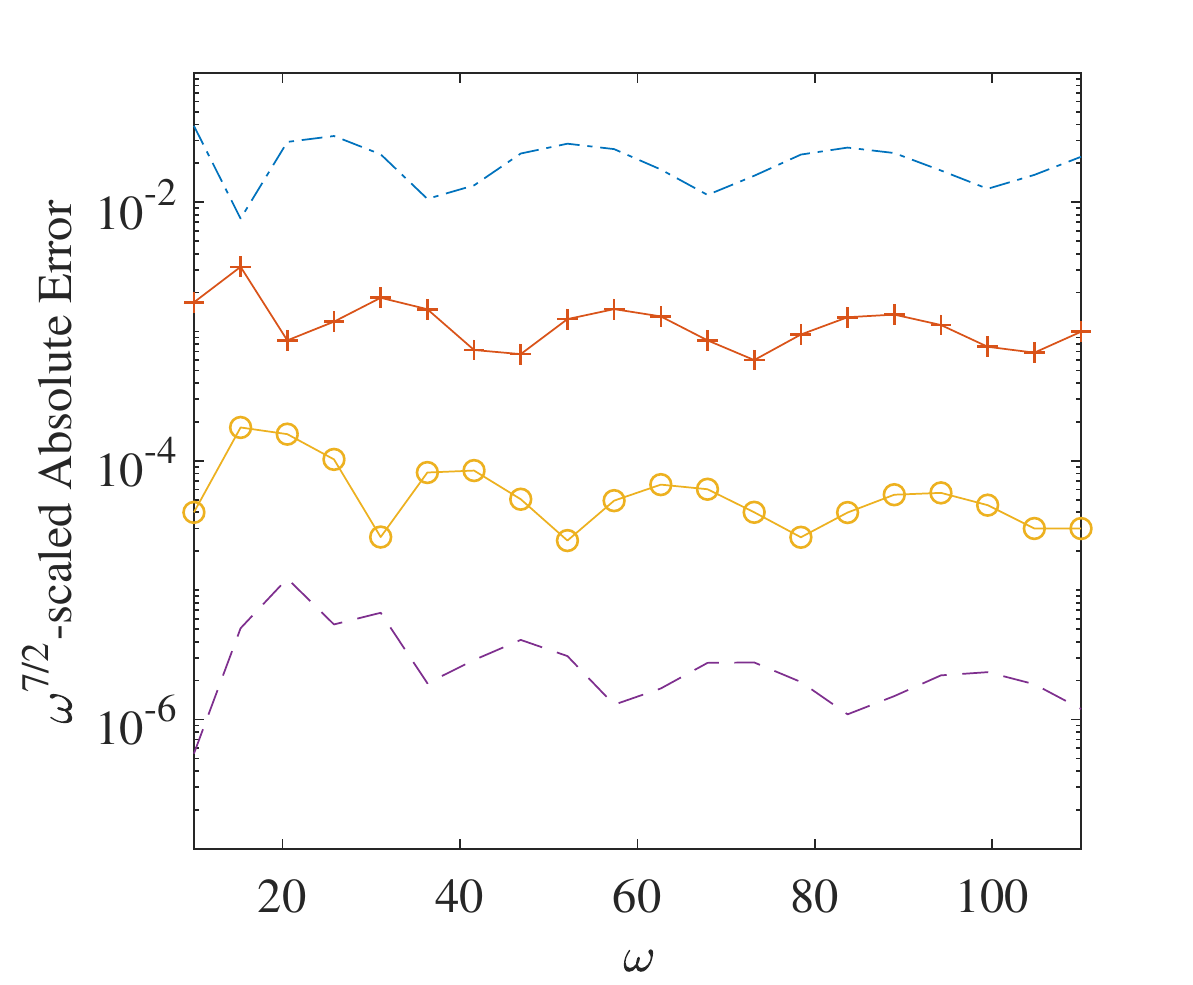}}
\caption{Scaled absolute errors of the new Levin method for the integral in Example 5.1 as a function of increasing $w$. Errors behave asymptotically as $\bO(w^{-s-1-\min\{1+\alpha,1\}})$ for different values of $s$ and $\alpha$.}
\label{fig1} 
\end{figure}

\begin{figure}
\centering
\subfloat[$s=0$,$\alpha=0.5$]{
\label{fig2:subfig:a} 
\includegraphics[width=3.0in]{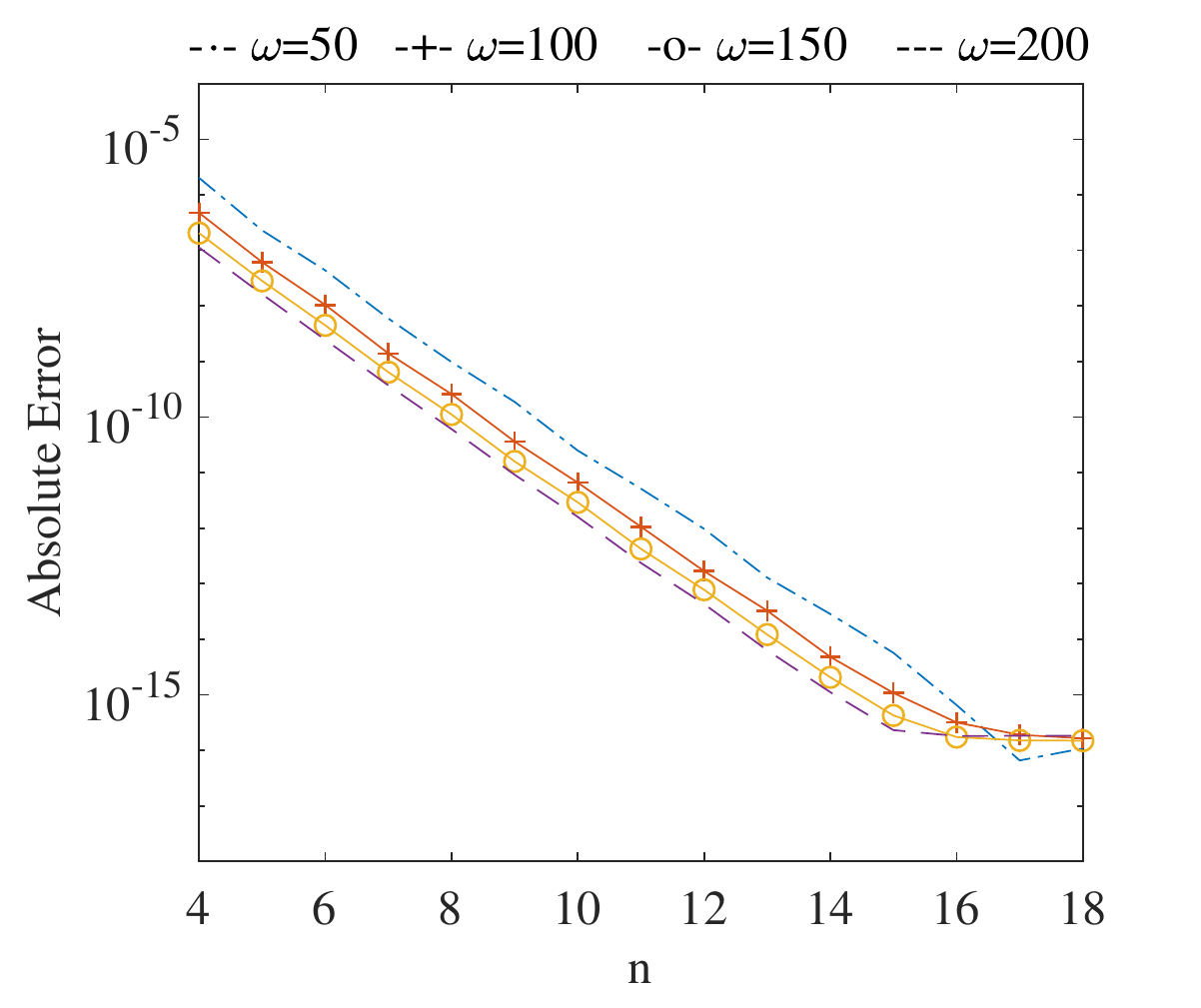}}
\hspace{0.1in}
\subfloat[$s=0$,$\alpha=-0.5$]{
\label{fig2:subfig:b} 
\includegraphics[width=3.0in]{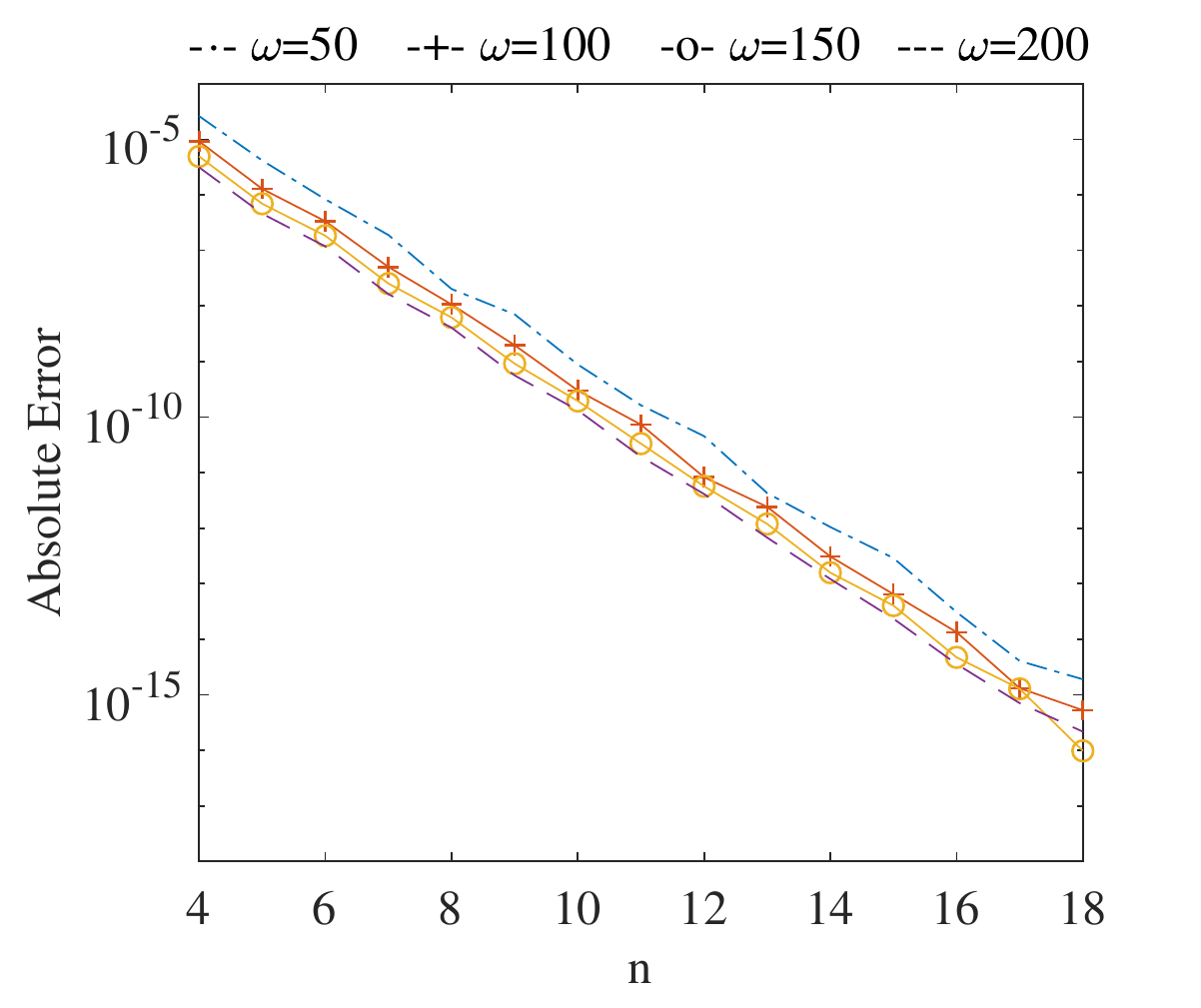}}

\subfloat[$s=1$,$\alpha=0.5$]{
\label{fig2:subfig:c} 
\includegraphics[width=3.0in]{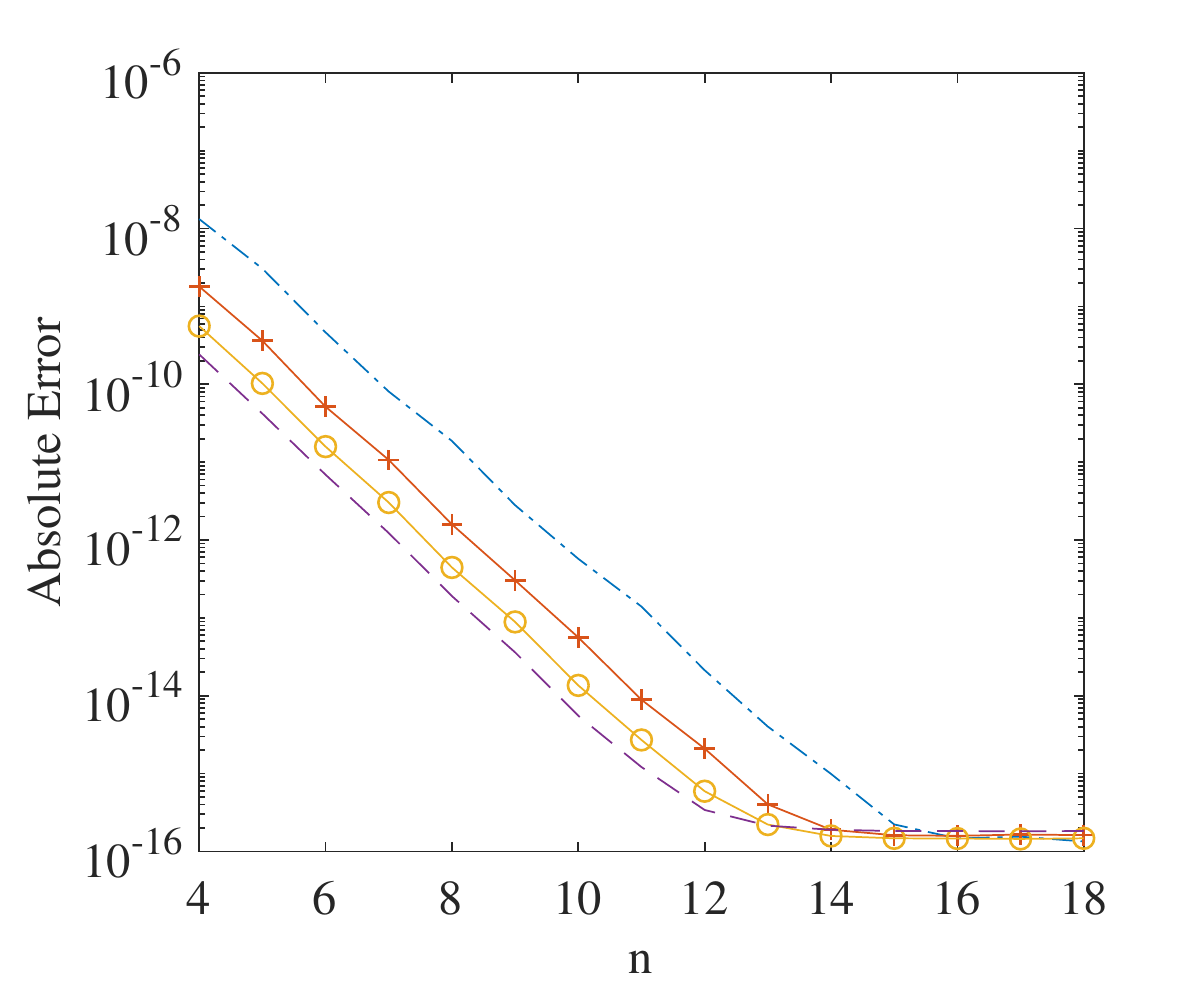}}
\hspace{0.1in}
\subfloat[$s=1$,$\alpha=-0.5$]{
\label{fig2:subfig:d}
\includegraphics[width=3.0in]{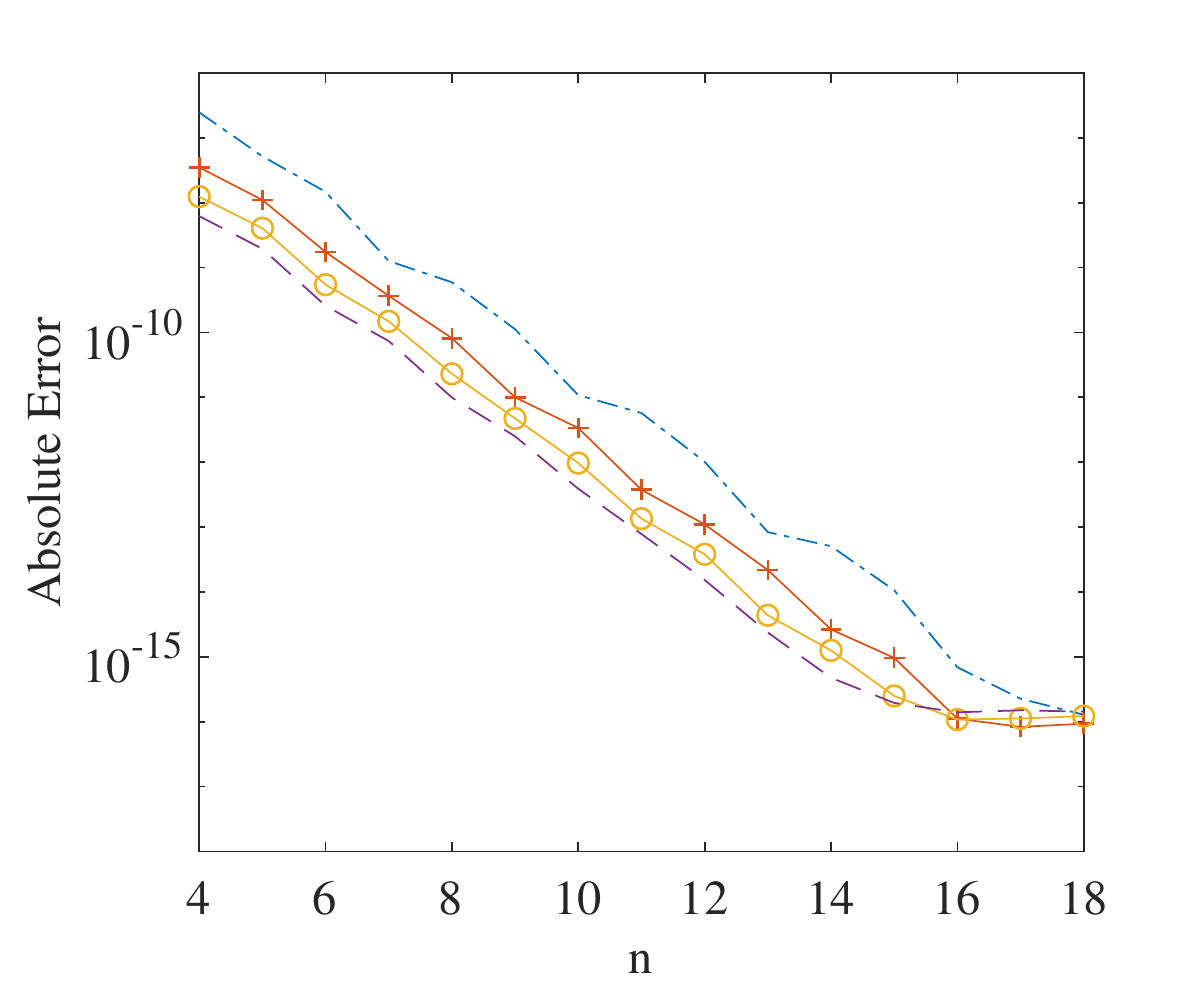}}

\subfloat[$s=2$,$\alpha=0.5$]{
\label{fig2:subfig:e}
\includegraphics[width=3.0in]{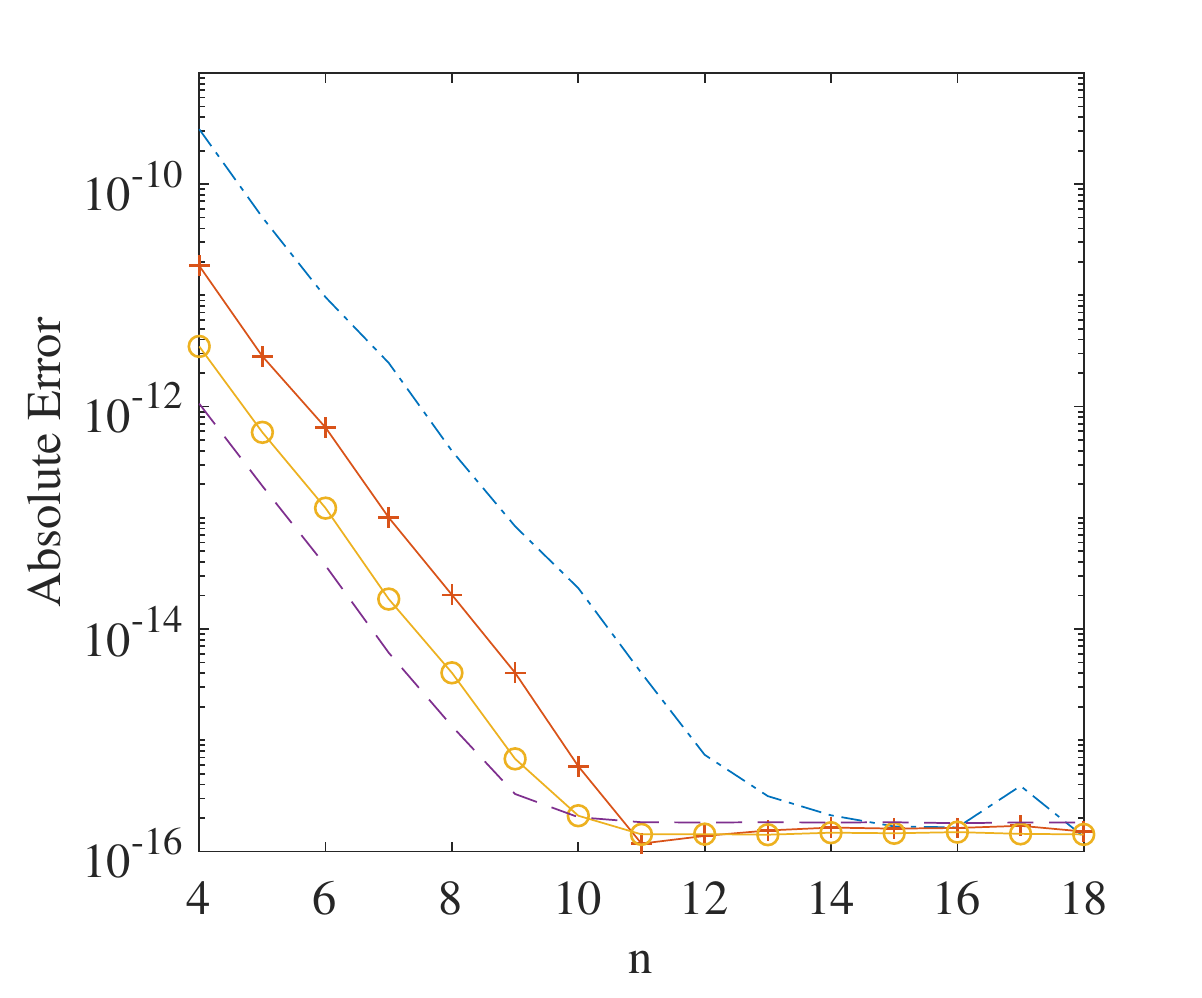}}
\hspace{0.1in}
\subfloat[$s=2$,$\alpha=-0.5$]{
\label{fig2:subfig:f}
\includegraphics[width=3.0in]{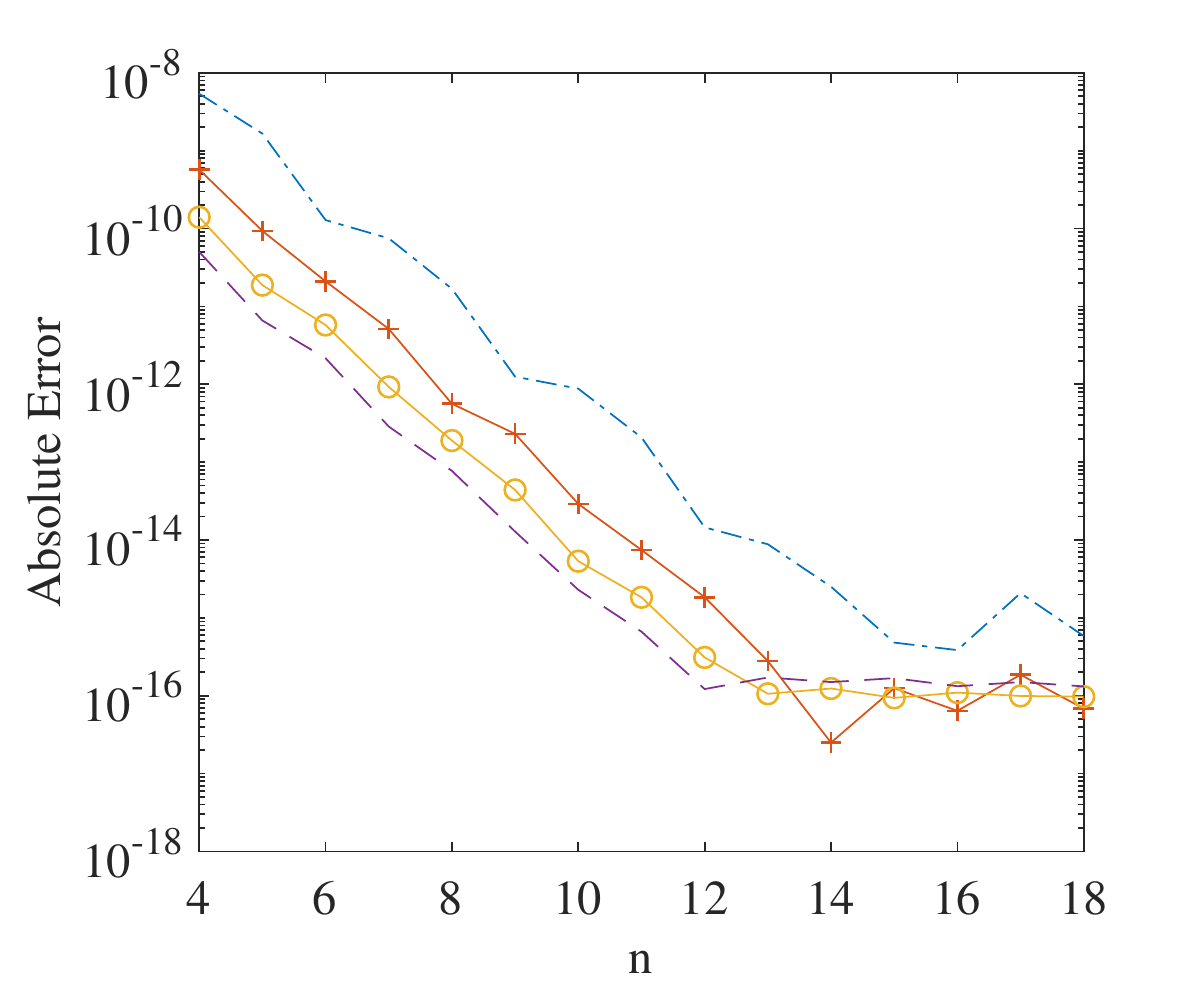}}
\caption{Absolute errors of the new Levin method for the integral in Example 5.1 as a function of increasing number of collocation points $n$. Exponential convergence is observed for different values of $s$ and $\alpha$.}
\label{fig2} 
\end{figure}
The Levin method is implemented based on the modified Chebyhev-Lobatto points. Figures \ref{fig1} and \ref{fig2} show numerical convergence for increasing frequency $w$ and for increasing number of collocation points. In Figure \ref{fig1}, $w^{s+1+\min\{1+\alpha,1\}}$-scaled absolute errors are plotted as a function of $w$. The lines are approximately straight, which confirms the asymptotic decay of the error at the rate of $w^{-1-s-\min\{1+\alpha,1\}}$. In Figure \ref{fig2}, convergence is shown as a function of $n$, the number of collocation points. Exponential convergence is observed, which levels off only when machine precision is reached. Errors decrease as the values of $w$ increase.

\begin{example}
In the second example, we compute the integral with algebraic and logarithmic singularities
\[
  \int_0^1\frac{1}{1+x^2}x^\alpha\log x e^{iwx}dx.
\]
\end{example}

We present in Figures \ref{fig4} and \ref{fig5} the similar results of numerical convergence for increasing frequency $w$ and for increasing number of collocation points . In Figure \ref{fig4} , $w^{s+1+\min\{1+\alpha,1\}}\delta_\alpha^{-1}(w)$-scaled absolute errors are plotted as a function of $w$. 
{\color{black}
The nearly straight lines confirm the asymptotic decay of the error}
 at the rate of $\delta_\alpha(w)w^{-1-s-\min\{1+\alpha,1\}}$. As a function of the number of collocation points, exponential convergence is observed in Figure \ref{fig5}. Errors also decrease as the values of $w$ increase.

\begin{figure}
\centering
\subfloat[$s=0$,$\alpha=0.5$]{
\label{fig4:subfig:a} 
\includegraphics[width=3.0in]{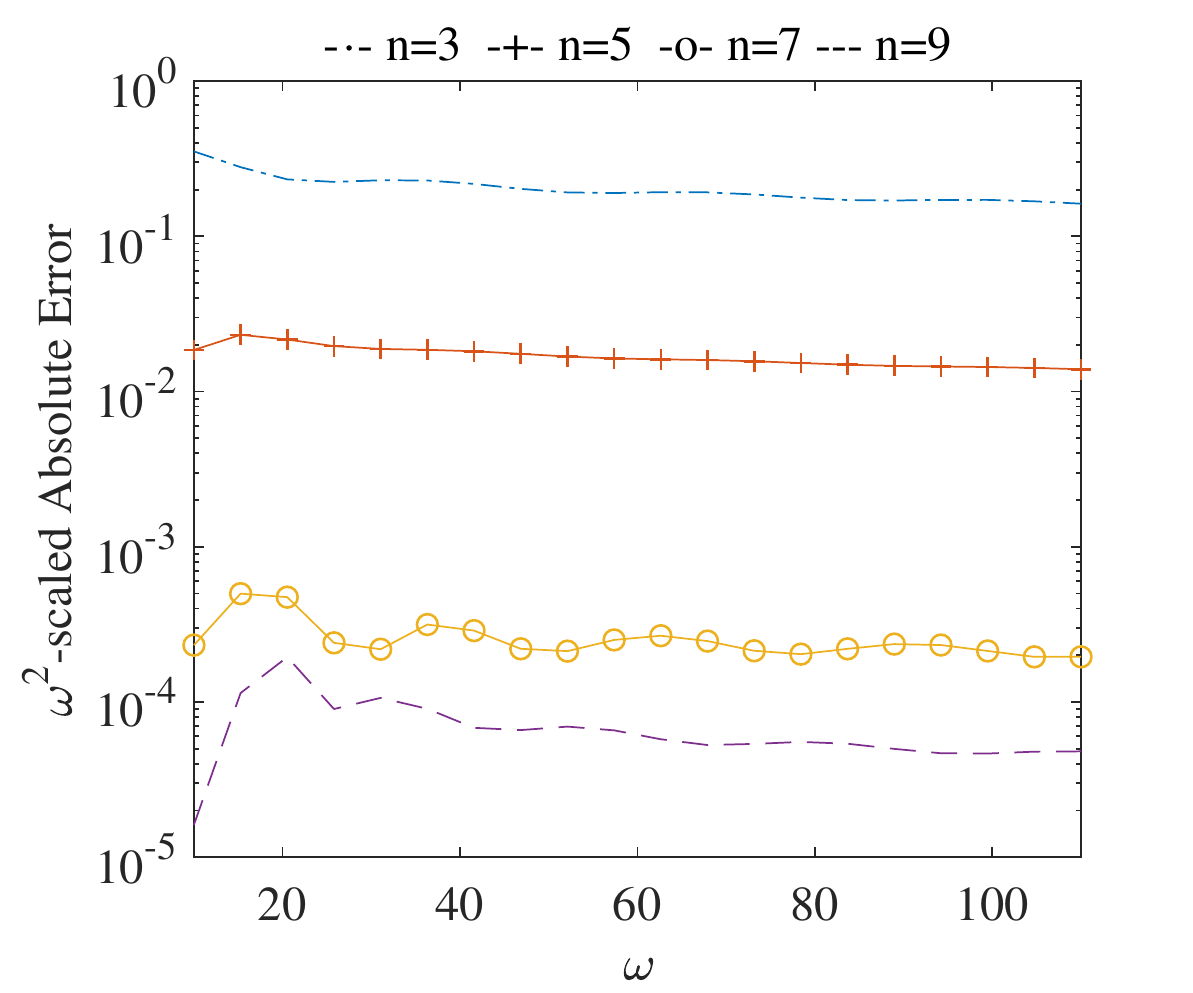}}
\hspace{0.1in}
\subfloat[$s=0$,$\alpha=-0.5$]{
\label{fig4:subfig:b} 
\includegraphics[width=3.0in]{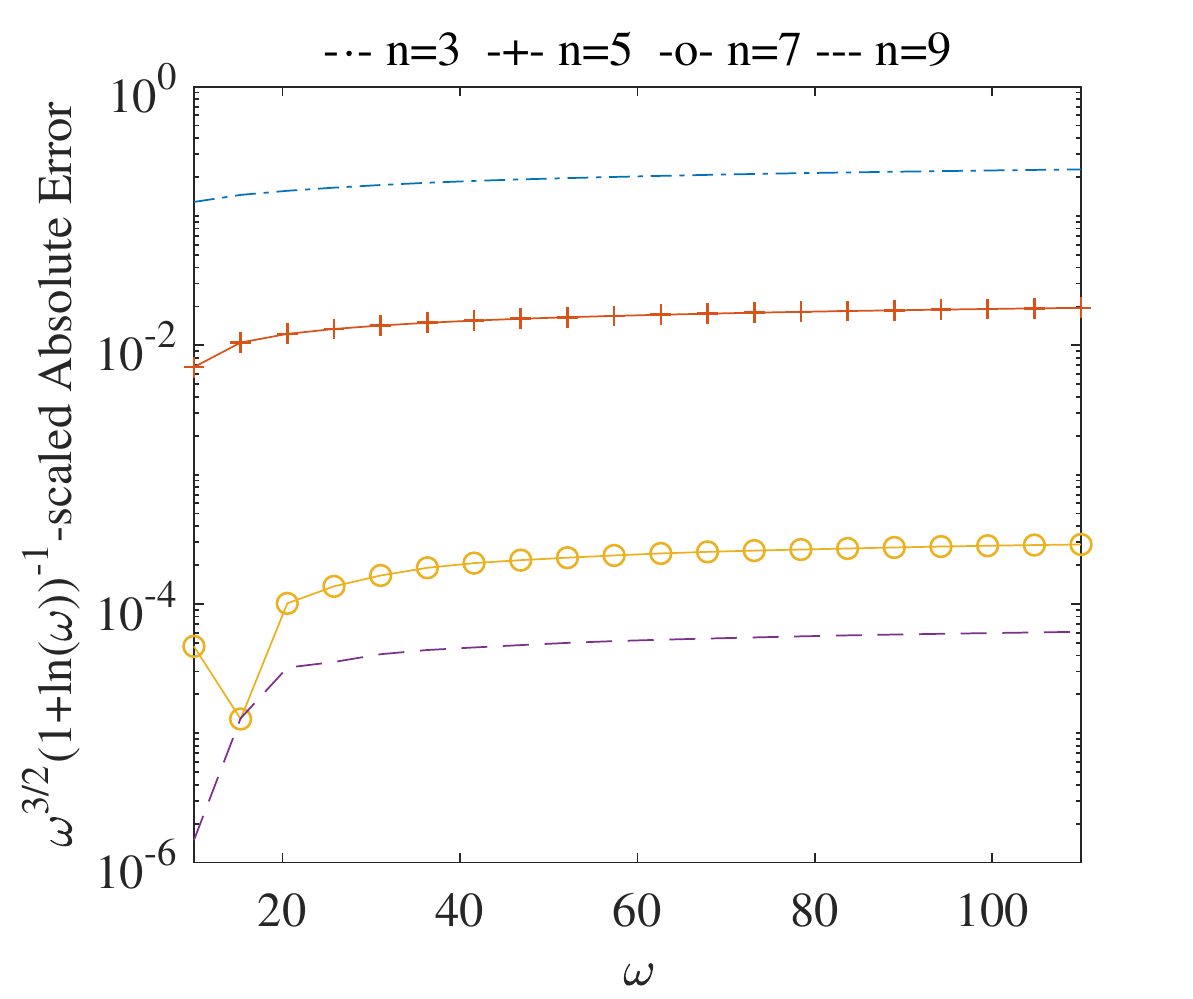}}

\subfloat[$s=1$,$\alpha=0.5$]{
\label{fig4:subfig:c}
\includegraphics[width=3.0in]{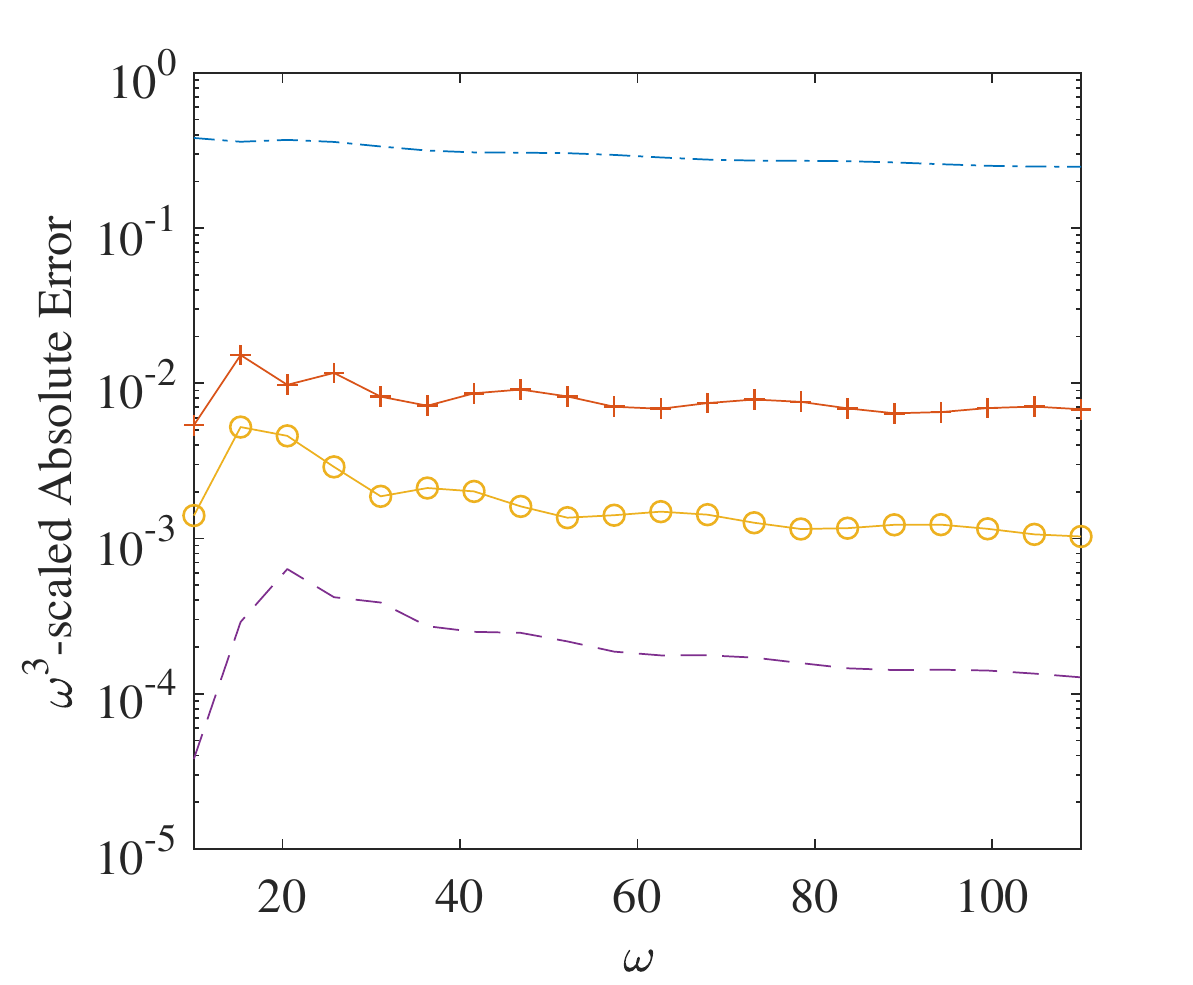}}
\hspace{0.1in}
\subfloat[$s=1$,$\alpha=-0.5$]{
\label{fig4:subfig:d} 
\includegraphics[width=3.0in]{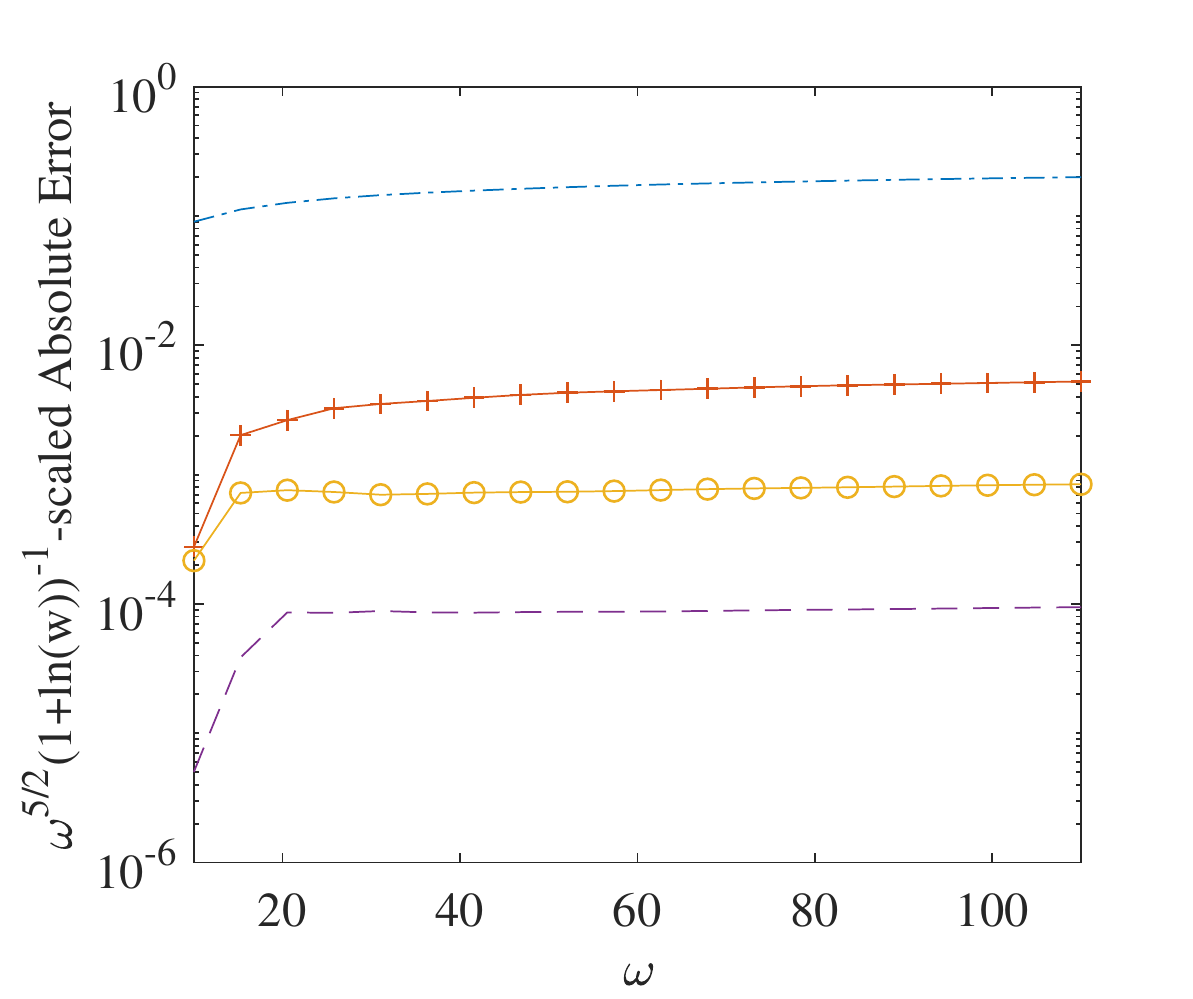}}

\subfloat[$s=2$,$\alpha=0.5$]{
\label{fig4:subfig:e} 
\includegraphics[width=3.0in]{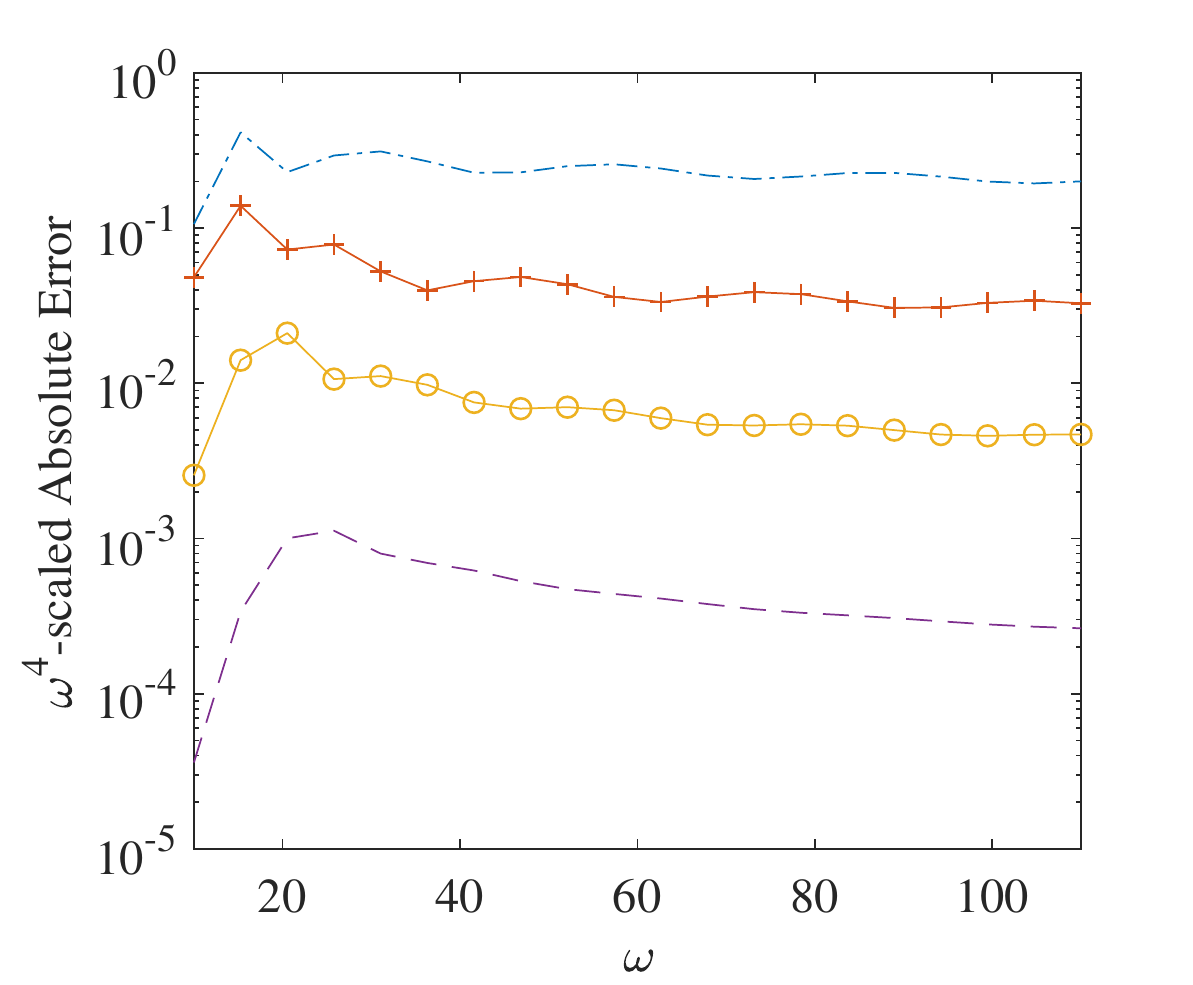}}
\hspace{0.1in}
\subfloat[$s=2$,$\alpha=-0.5$]{
\label{fig4:subfig:f} 
\includegraphics[width=3.0in]{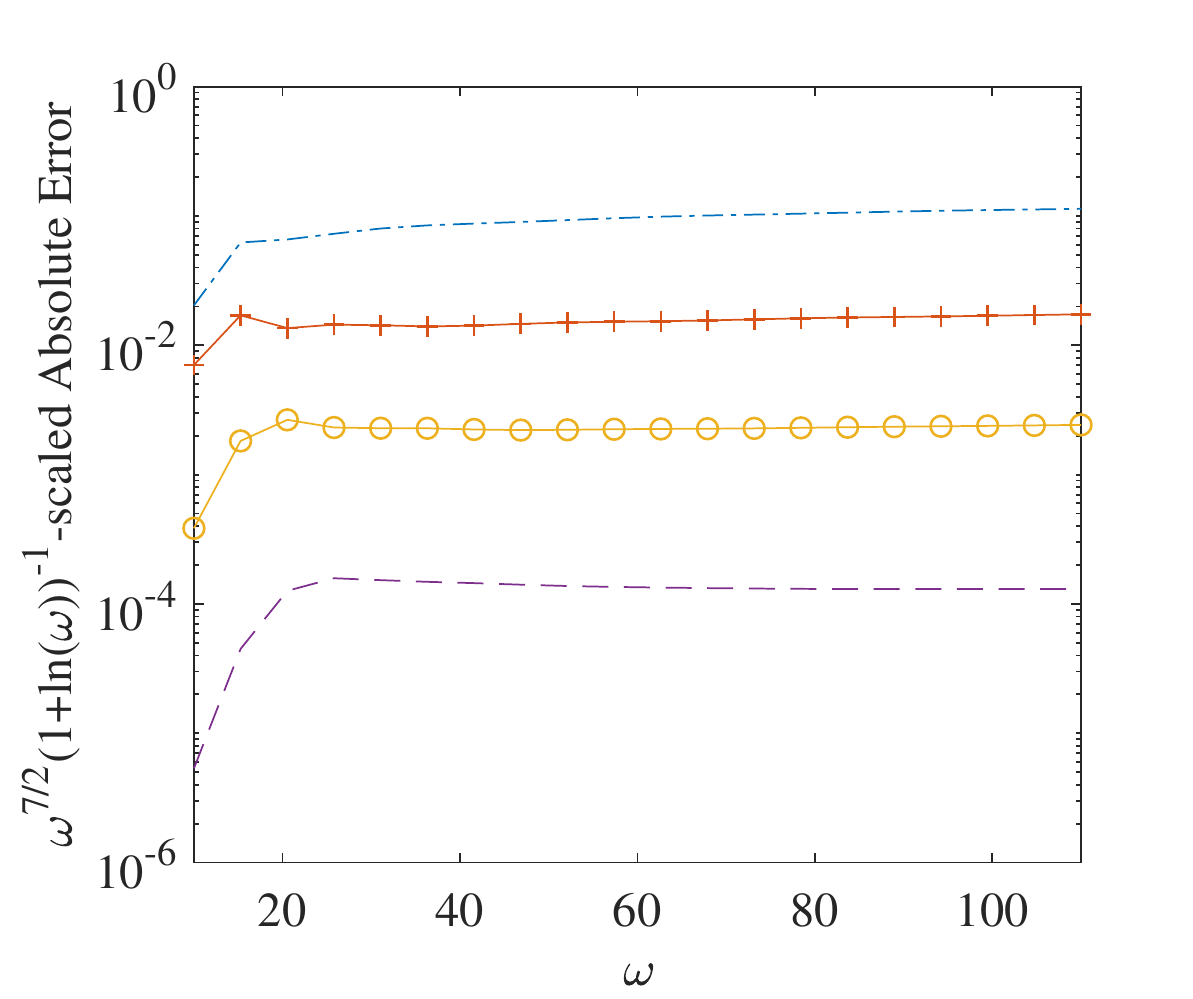}}
\caption{Scaled absolute errors of the new Levin method for the integral in Example 5.2 as a function of increasing $w$. Errors behave asymptotically as $\bO(\delta_\alpha (w)w^{-s-1-\min\{1+\alpha,1\}})$ for different values of $s$ and $\alpha$.}
\label{fig4} 
\end{figure}

\begin{figure}
\centering
\subfloat[$s=0$,$\alpha=0.5$]{
\label{fig5:subfig:a}
\includegraphics[width=3.0in]{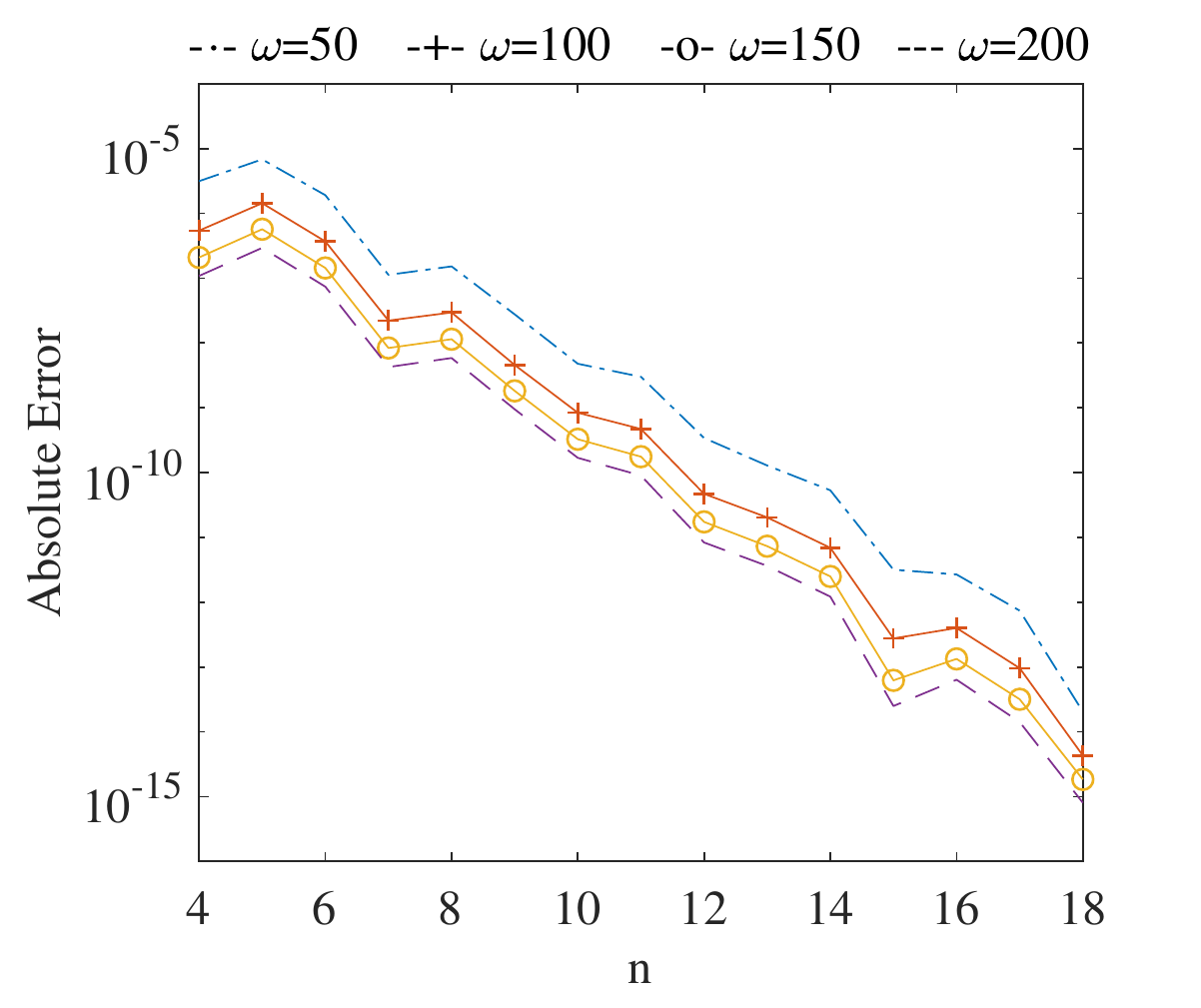}}
\hspace{0.1in}
\subfloat[$s=0$,$\alpha=-0.5$]{
\label{fig5:subfig:b} 
\includegraphics[width=3.0in]{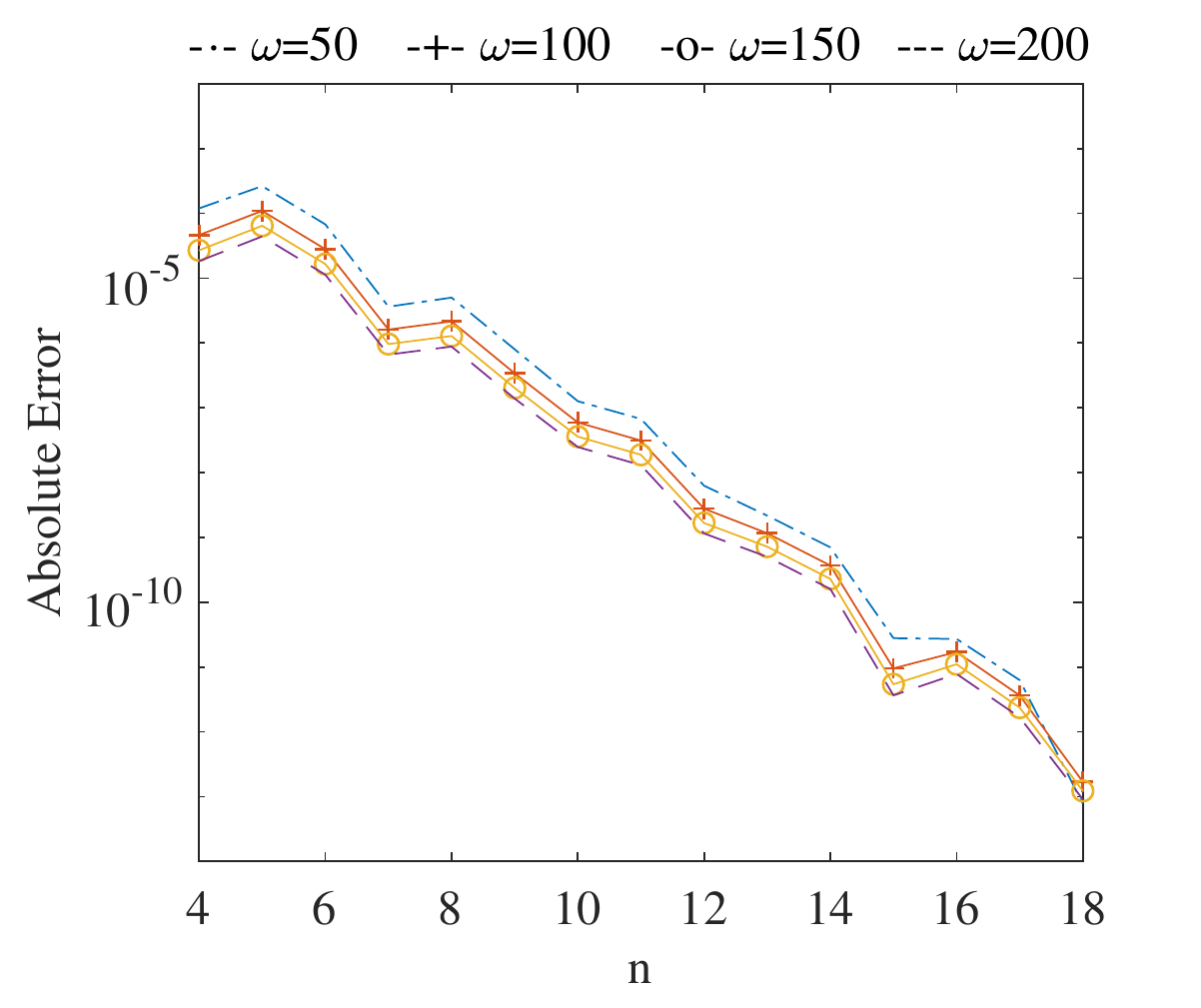}}

\subfloat[$s=1$,$\alpha=0.5$]{
\label{fig52:subfig:c}
\includegraphics[width=3.0in]{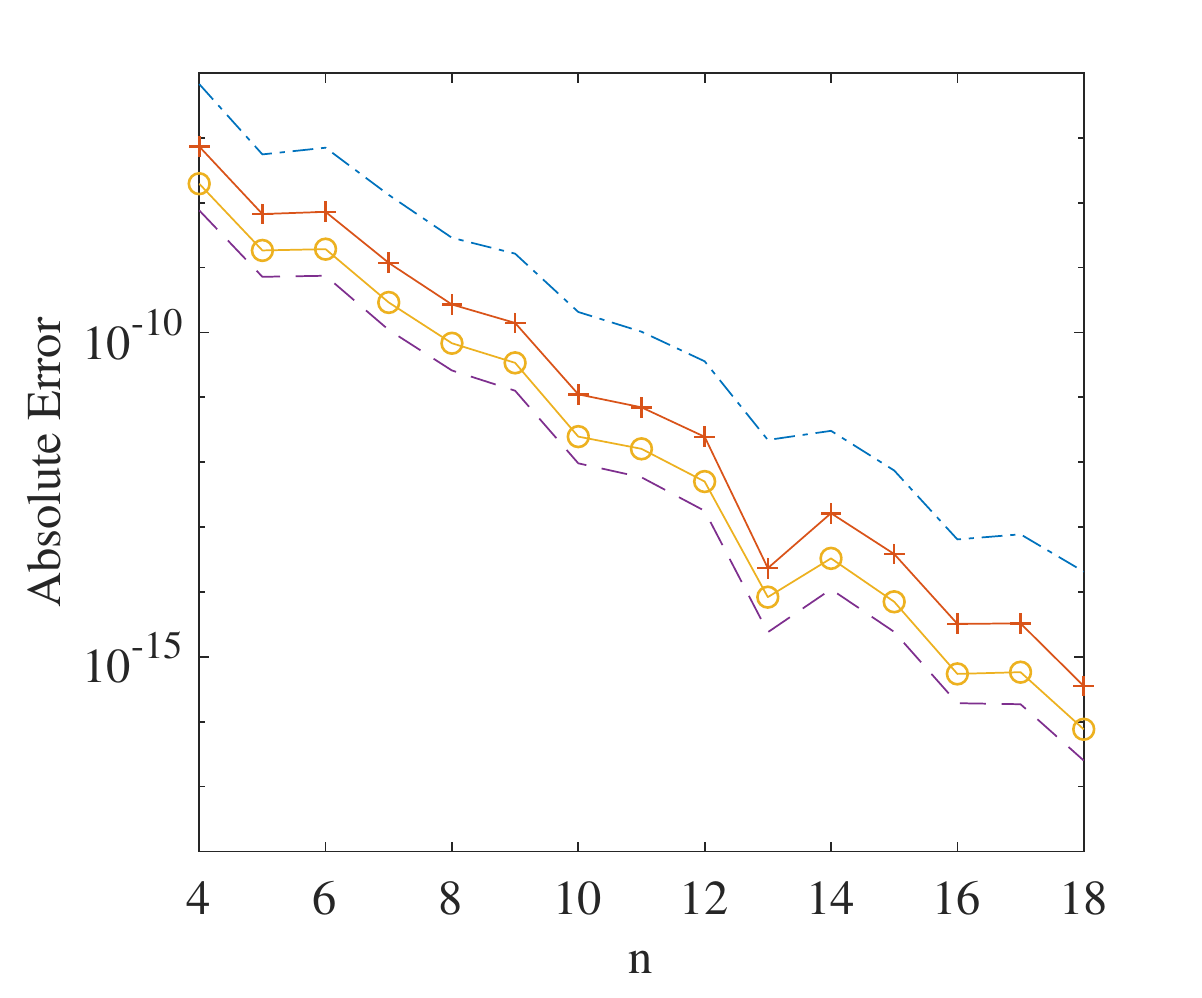}}
\hspace{0.1in}
\subfloat[$s=1$,$\alpha=-0.5$]{
\label{fig5:subfig:d} 
\includegraphics[width=3.0in]{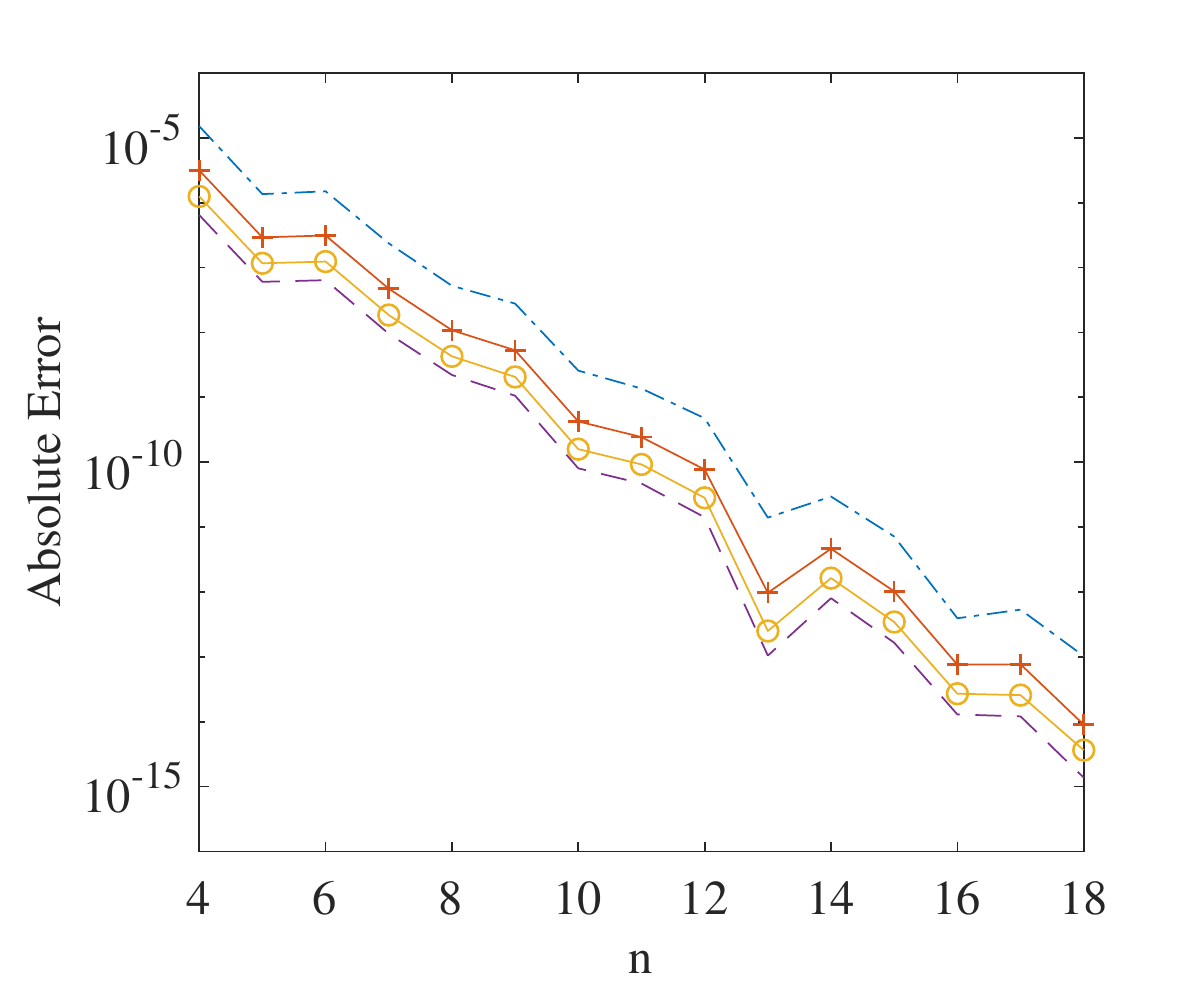}}

\subfloat[$s=2$,$\alpha=0.5$]{
\label{fig5:subfig:e} 
\includegraphics[width=3.0in]{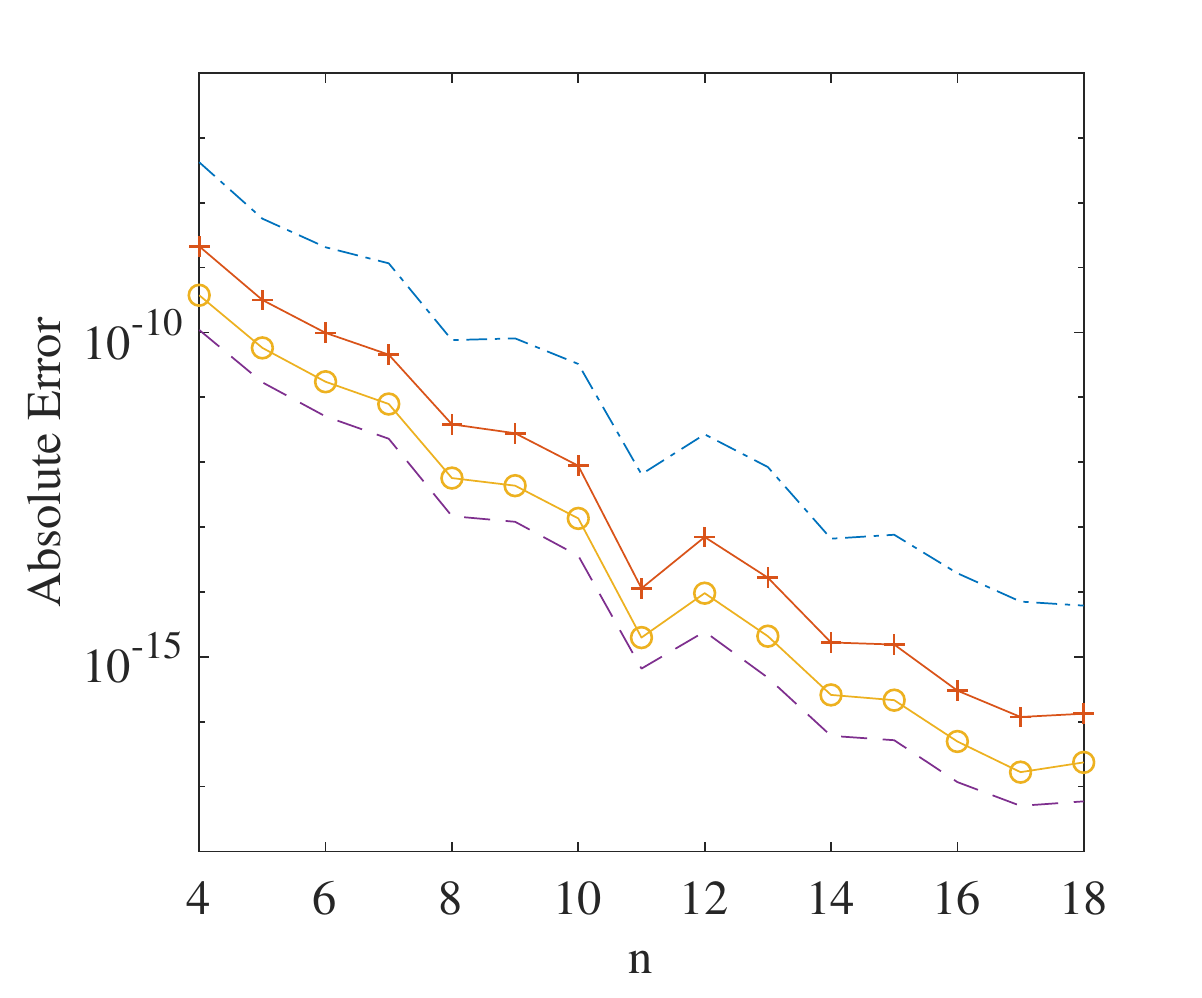}}
\hspace{0.1in}
\subfloat[$s=2$,$\alpha=-0.5$]{
\label{fig5:subfig:f} 
\includegraphics[width=3.0in]{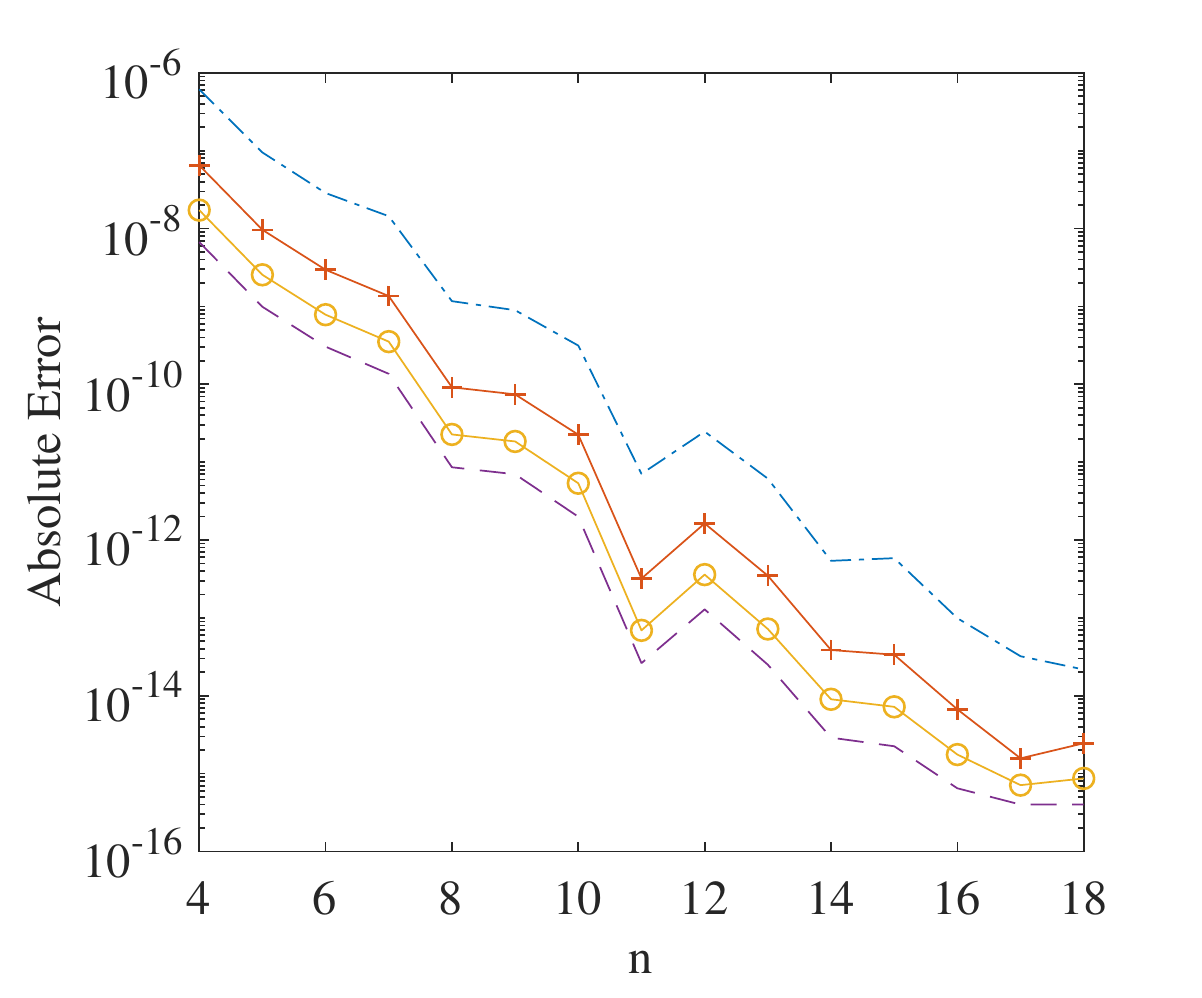}}
\caption{Absolute errors of the new Levin method for the integral in Example 5.2 as a function of increasing number of collocation points $n$. Exponential convergence is observed for different values of $s$ and $\alpha$.}
\label{fig5} 
\end{figure}

\begin{example}
To compare the convergence of the new Levin methods and the corresponding Filon-type methods, we consider two integrals with a non-linear oscillator,
\begin{equation}
\int_0^1\frac{1}{1+x^2}x^\alpha e^{iw(x^2+x+1)}dx\;\text{and}\;  \int_0^1\frac{1}{1+x^2}x^\alpha \log x e^{iw(x^2+x+1)}dx.
\end{equation}

\end{example}

Filon-type methods are implemented based on the basis $\Psi_M$ while Levin methods are based on Chebyshev polynomials. 
{\color{black}
Both adopt the modified Chebyhev-Lobatto points as collocation points.} Tables \ref{fig7} and \ref{fig8} show the absolute errors of integrals $\int_0^1\frac{1}{1+x^2}x^\alpha e^{iw(x^2+x+1)}dx$ and $\int_0^1\frac{1}{1+x^2}x^\alpha \log x e^{iw(x^2+x+1)}dx$, respectively. We fix $w=100$ and $\alpha=0.5$. It is shown that the errors of Levin methods are much smaller than those of Filon-type methods. 
{\color{black}
Hence, the new Levin methods outperform the Filon-type methods} when computing an integral with a nonlinear oscillator.

\begin{table}[htb]
\begin{center}
\small
\caption{Absolute errors of the Levin method and Filon-type method for integral $\int_{0}^{1}\frac{1}{1+x^2}x^\alpha e^{iw(x^2+x+1)}dx$ with $w=100$ and $\alpha=0.5$.}
\label{fig7}
\begin{tabular}{c|ccc|ccc}
\hline
\multirow{2}*{\centering $n$}  & \multicolumn{3}{c}{Levin ( $Q_{w,\alpha,n}^{L,s}[f]$ )}  &  \multicolumn{3}{|c}{Filon-type ( $Q_{w,\alpha,n}^{F,s}[f]$ )} \\
\cline{2-7}
& $s=0$  & $s=1$  &  $s=2$   & $s=0$  & $s=1$  &  $s=2$   \\
\hline
4&$1.5382e-05$&$2.6363e-07$&$1.2572e-08$&$4.3048e-05$&$2.1433e-06$&$2.1830e-07$\\ 
6&$2.3171e-06$&$2.5048e-08$&$1.6132e-09$&$2.7080e-05$&$1.4631e-06$&$1.5978e-07$\\ 
8&$3.0090e-07$&$1.4410e-09$&$1.8206e-10$&$1.6425e-05$&$9.5278e-07$&$1.0084e-07$\\ 
10&$2.9168e-08$&$2.1253e-10$&$1.8362e-11$&$9.4919e-06$&$5.9085e-07$&$6.0610e-08$\\ 
12&$2.5673e-09$&$3.8829e-11$&$1.7950e-12$&$5.3341e-06$&$3.5403e-07$&$3.6018e-08$\\ 
14&$1.8243e-10$&$4.9531e-12$&$1.9541e-13$&$2.9614e-06$&$2.0881e-07$&$8.6379e-09$\\ 
\hline
\end{tabular}
\end{center}
\end{table}

\begin{table}[htb]
\begin{center}
\small
\caption{Absolute errors of the Levin method and Filon-type method for integral $\int_{0}^{1}\frac{1}{1+x^2}x^\alpha\log(x)e^{iw(x^2+x+1)}dx$ with $w=100$ and $\alpha=0.5$.}
\label{fig8} 
\begin{tabular}{c|ccc|ccc}
\hline
\multirow{2}*{\centering $n$}  & \multicolumn{3}{c}{Levin ( $Q_{w,n}^{L,s}[f]$ )}  &  \multicolumn{3}{|c}{Filon-type ( $Q_{w,n}^{F,s}[f]$ )} \\
\cline{2-7}
& $s=0$  & $s=1$  &  $s=2$   & $s=0$  & $s=1$  &  $s=2$   \\
\hline
4&$2.2974e-05$&$9.0885e-07$&$3.1955e-08$&$5.0237e-05$&$4.3859e-06$&$1.7800e-07$\\ 
6&$2.4346e-06$&$1.1992e-07$&$3.6359e-09$&$3.0141e-05$&$2.5481e-06$&$1.3454e-07$\\ 
8&$1.2843e-07$&$1.2403e-08$&$4.0700e-10$&$1.9505e-05$&$1.4536e-06$&$1.1448e-07$\\ 
10&$6.3650e-09$&$1.1533e-09$&$4.7044e-11$&$1.1996e-05$&$8.2261e-07$&$7.9733e-08$\\ 
12&$2.3501e-09$&$9.4296e-11$&$5.2921e-12$&$6.9970e-06$&$4.7287e-07$&$4.9349e-08$\\ 
14&$3.5590e-10$&$8.0273e-12$&$5.5103e-13$&$3.9501e-06$&$2.7718e-07$&$1.1227e-08$\\ 
\hline
\end{tabular}
\end{center}
\end{table}

\begin{example}
In the final example, we show the efficiency of the new collocation method in Section 4 by recomputing the integral in Example \ref{example1} and comparing relative errors and CPU time with those of the CMFP.
\end{example}
To this end, we simply recall the quadrature formulas of the CMFP.
The moment-free Filon method in \cite{XIANG2007} approximates the integral $\int_a^b f(x)e^{iwg(x)}dx$ by
\[
Q^{[a,b],MF}_{w,m}[f,g]:=\int_{g(a)}^{g(b)} p_n(x)e^{iwx}dx,
\]
where $p_n$ is a polynomial of degree $n-1$ 
{\color{black}
that interpolates $\left[(f/g')\circ g^{-1}\right]$ at $g(t_j),j=0,1,\ldots,m$ and $t_j, j=0,1,\ldots,m$ comprise a set of distinguishing points on $[a,b]$. The composite moment-free Filon-type rules used in CMFP read}
\[
Q^{[a,b],CMF}_{w,n,m}[f,g]:=\sum_{j=1}^{n} Q^{[x_{j-1},x_j],MF}_{w,m}[f,g]\; \text{with}\; x_j=a+\frac{j}{n}(b-a), j=0,1,\ldots,n.
\]
The Gauss-Legendre quadrature rule for integral $\int_a^b f(x)dx$ is given by
\[
Q_m^{[a,b],GL}[f]:=\frac{b-a}{2}\sum_{j=1}^m w_j f\left(\frac{(b-a)t_j+b+a}{2}\right),
\]
where $w_j$ and $t_j$ are the standard weights and points of the Gauss-Legendre rule on the domain $[-1,1]$.
Suppose for a non-negative integer $r$ that the function $g\in C^{r+1}[0,1]$ has a single stationary point at zero and satisfies $g^{(j)}(0)=0$ for $j=1,\ldots,r$ and $g^{(r+1)}(x)\neq 0$ for $x\in[0,1]$. 
{\color{black}
Letting $\sigma_r:=\|g^{(r+1)}\|_\infty/(r+1)!$,} $w_r:=\max\{k\sigma_r,k\}$, and $\lambda_r:=w_r^{-1/(r+1)}$,
the CMFP method for integral $\int_0^1 f(x)e^{ikg(x)}dx$ is established by
\[
Q_{w,n,s,m_1,m_2}^{CMFP}[f,g]:=\lambda_r\sum_{j=1}^{s-1}Q_{m_1}^{[x_j,x_{j+1}],GL}[\varphi]+ \sum_{j=1}^{n} Q^{[y_{j-1},y_{j}],CMF}_{w,N_j,m_2}[f,g],
\]
where $x_0=0$, $x_j=(j/s)^p, p=(2m+1)/(1+\mu)$, $\mu$ is the index of singularity of $f$, $j=1,2,\ldots,s$, $\varphi(x)=f(\lambda_r x)e^{iwg(\lambda_r x)}$,
$y_j=w_r^{(j-n)/n/(r+1)}, j=0,1,\ldots,n$, $N_j=\lceil q_j^{m/(m-1)}\rceil$, $q_j=\max\{|g'(y_{j-1})|,|g'(y_j)|\}x_{j-1}/g(x_{j-1})$, $j=1,2,\ldots,n$, and $\nu$ is the index of singularity of $\left[(f/g')\circ g^{-1}\right]$. When $f$ has only the logarithmic singularity, the value of $\mu$ is set to 0.
 When $g(x)=x$, $Q_{w,n,s,m_1,m_2}^{CMFP}[f,g]$ is shortened to $Q_{w,n,s,m_1,m_2}^{CMFP}[f]$.

 The comparisons of relative errors and CPU time are shown in Figure \ref{fig3}, between the new Levin method by $Q_{-w,\alpha,n}^{L,0}[f]$ and the CMFP by $Q_{-w,n_1,n_1,4,4}^{CMFP}[f]$ with $r=0$ and $\sigma_r=1$, where $\alpha=-0.5$, $w=100$ or $1000$, $n$ is set to $[4,6,\ldots,36]$, and $n_1=2^{(n-2)/2}$. 
{\color{black}
The left-hand panel of Figure \ref{fig3} shows} that errors of the Levin method decrease faster and errors of the CMFP increase when the number of points is large enough (seen in the ellipse of dashes).
 Hence, the new Levin method is more stable. 
 {\color{black}
 The right-hand panel shows that the Levin method takes less time to attain machine precision and the CPU time grows slower. This is because the proposed method has reached superalgebraic convergence.}

\begin{figure}
\centering
\includegraphics[width=3in]{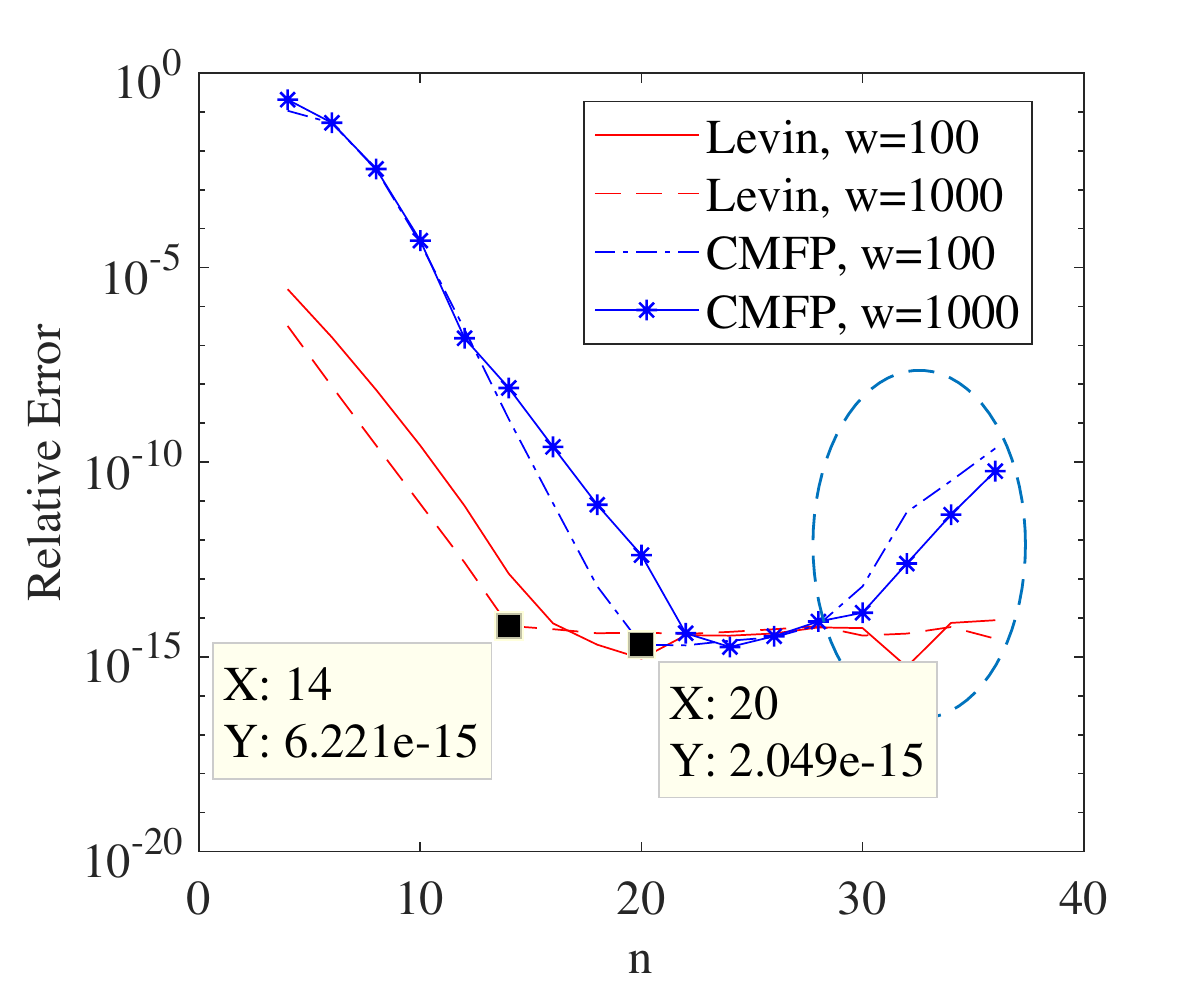}
\hspace{.31in}%
\includegraphics[width=3in]{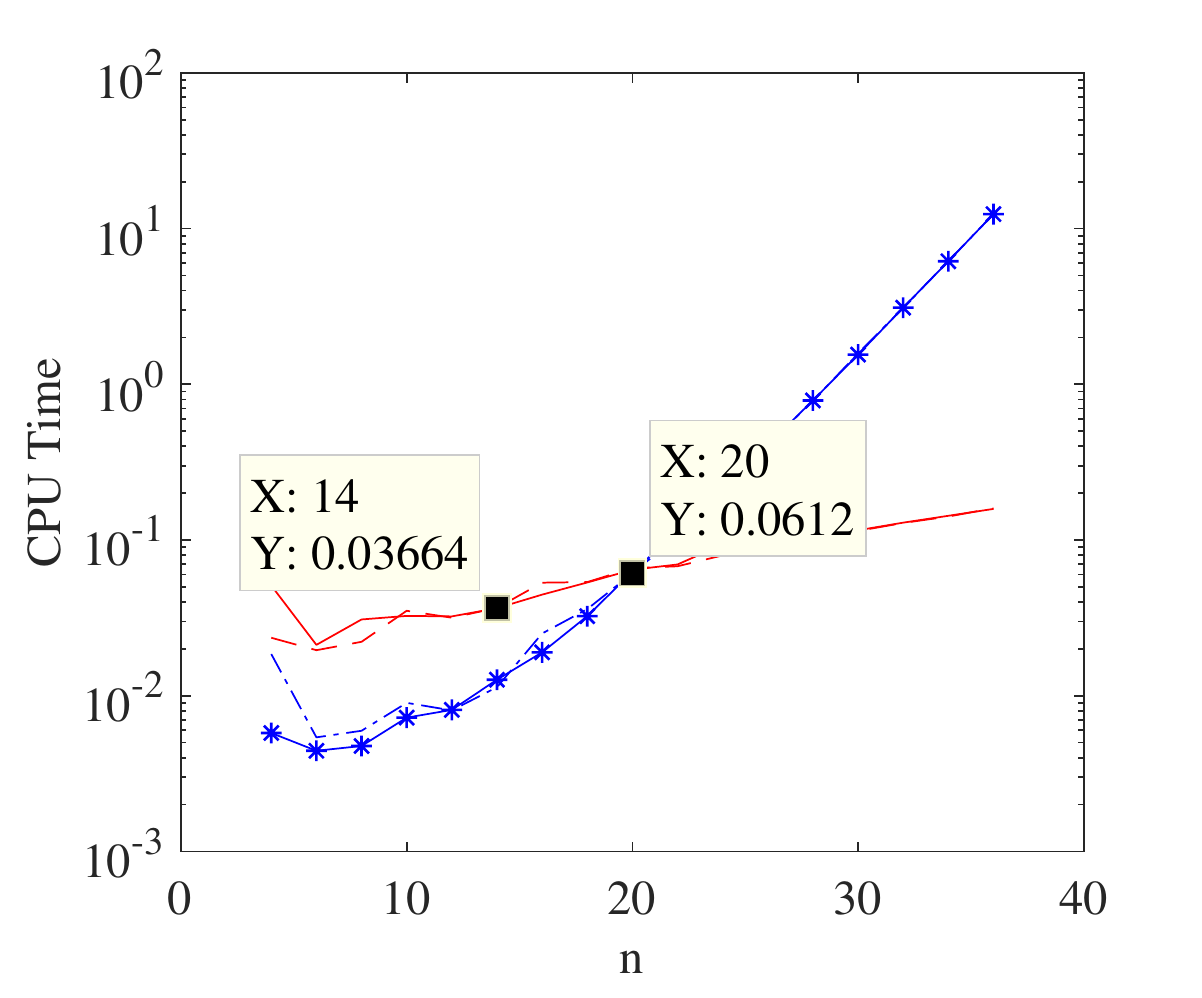}
\caption{Comparison of the relative errors (left) and the CPU time (right) in computing $\int_0^1 f(x)x^\alpha e^{-iwx}dx$ by the new Levin method ($Q_{-w,\alpha,n}^{L,0}[f]$) and CMFP ($Q_{-w,n_1,n_1,4,4}^{CMFP}[f] $), where $\alpha=-0.5$, $n$ is set to be $[4,6,\ldots,36]$, and $n_1=2^{(n-2)/2}$. }
\label{fig3}
\end{figure}

\begin{appendices}
\renewcommand{\theequation}{\thesection.\arabic{equation}}
\section{Proof of Lemma \ref{lemma3}}
To prove Lemma \ref{lemma3}, we first present a useful result in the next lemma.
\begin{lemma}\label{lemma01}
Let $w\in\aR$ be a parameter and assume that $g\in C^1[0,a]$ is a function independent of $w$ satisfying $g(0)=0$ and $g'(x)>0$ for $x\in[0,a]$. If $f\in C[0,a]$ and $q$ is a solution of the ODE,
    \begin{equation*}
        g(x)q'(x)+[k+\alpha+iwg(x)]g'(x)q(x)=f(x), k=1,2,\ldots,
    \end{equation*}
    with the initial condition $q(0)=0$, then there exists a constant $C$ independent of $w$ such that
    \begin{equation*}
        \|q\|_\infty\leq C\|f\|_\infty.
    \end{equation*}
\end{lemma}
\begin{proof} Multiplying both sides with the term $g^{k+\alpha-1}(x)e^{iwg(x)}$, and integrating over the domain $[0,x]$, we derive
\begin{equation*}
    q(x)=g^{-k-\alpha}(x)e^{-iwg(x)}\int_0^x f(t)g^{k+\alpha-1}(t) e^{iwg(t)}dt.
\end{equation*}
Taking the absolute value,
\begin{equation*}
    |q(x)|
    \leq\|f\|_\infty \left| g^{-k-\alpha}(x) \right| \int_0^x \left|g^{k+\alpha-1}(t)\right| dt.
\end{equation*}
For any $x\in[0,a)$, there exists $\xi\in[0,a)$ such that $g(x)/x=g'(\xi)$. Since $g'(x)>0$ on $[0,a]$ and $g'\in C[0,a]$, there exist two positive constants $C_1$ and $C_2$ depending on $g$ and $a$ such that $C_1<|g(x)/x|<C_2$, $x\in[0,a]$. Then, it is obtained that
\begin{equation*}
    |q(x)|\leq\|f\|_\infty C_1^{-k-\alpha}C_2^{k+\alpha-1}x^{-k-\alpha}\int_0^x t^{k+\alpha-1}dt=\frac{C_2^{k+\alpha-1}}{(k+\alpha)C_1^{k+\alpha}} \|f\|_\infty.
\end{equation*}

The proof is finished by setting $C=\frac{C_2^{k+\alpha-1}}{(k+\alpha)C_1^{k+\alpha}}$.
\end{proof}

{\color{black} 
Now we are ready to prove Lemma \ref{lemma3}.}

\begin{proofspecial}
     Rewriting the ODE  \eqref{gq1eq} as
     \begin{equation}\label{flsode2}
        c_0+g(x)q_1(x)=\frac{1}{iwg'(x)}\left(f_1(x)-g(x)q_1'(x)-(1+\alpha)q_1(x)\right),
     \end{equation}
we then generate a sequence of successive approximations. 
{\color{black}
Setting the initial settings as $q_1^{[0]}\equiv0$,} we obtain, for $k\geq1$,
\begin{equation*}
\begin{split}
     \Phi^{[k]}(x)&:= \frac{1}{iwg'(x)}\left(f_1(x)-g(x)\bD q_1^{[k-1]}(x)-(1+\alpha)q_1^{[k-1]}(x)\right), \\
     c^{[k]}_0&:=\Phi^{[k]}(0),\\
     q_1^{[k]}(x)&:=\frac{1}{g(x)}\left(\Phi^{[k]}(x)-c^{[k]}_0\right).
\end{split}
\end{equation*}
Since $f_1\in C^{2n+1}[0,a]$ and $g\in C^{2n+2}[0,a]$ with $g'(x)>0$ for $x\in[0,a]$, it can be obtained by induction that $\Phi^{[k]}, q_1^{[k]}\in C^{2n-k+2}[0,a]$ and
\begin{equation}\label{qn2}
    \left\|\bD^{m}\Phi^{[k]}\right\|_\infty=\bO(w^{-1}),\max_{j}\left|c^{[k]}_j\right|=\bO(w^{-1}) \; \text{and}\; \left\|\bD^{m}q_1^{[k]}\right\|_\infty=\bO(w^{-1}),
\end{equation}
for $m=0,1,\ldots,2n-k+1,\; k=1,\ldots,n+1$. Therefore, functions $q_1^{[n+1]}$ and $c^{[n+1]}_0$ possess the desired property \eqref{property} and satisfy
\begin{equation*}
\begin{split}
    &iwg'(x)c^{[n+1]}_0+ g(x)\bD q_1^{[n+1]}(x)+[1+\alpha+iwg(x)]g'(x)q_1^{[n+1]}(x) \\
    &\quad\quad =f_1(x)+g(x)\bD\left(q_1^{[n+1]}(x)- q_1^{[n]}(x)\right)+(1+\alpha)g'(x)\left(q_1^{[n+1]}(x)- q_1^{[n]}(x) \right).
\end{split}
\end{equation*}

From the relation
\begin{equation*}
    \Phi^{[k+1]}(x)-\Phi^{[k]}(x)=-\frac{1}{iwg'(x)} \left[ g(x)\bD\left(q_1^{[k]}(x)- q_1^{[k-1]}(x)\right)+(1+\alpha)g'(x)\left(q_1^{[k]}(x)- q_1^{[k-1]}(x) \right) \right], \; k\geq1,
\end{equation*}
it is derived by induction that
\begin{equation*}
    \left\|g(x)\bD\left(q_1^{[n+1]}- q_1^{[n]}\right)(x) +(1+\alpha)g'(x)\left(q_1^{[n+1]}(x)- q_1^{[n]}(x) \right) \right\|_\infty=\bO(w^{-n-1}).
\end{equation*}
Now we define $q_1$ to be the solution of the ODE \eqref{gq1eq} with $c_0=c^{[n+1]}_0$ and initial condition $q_1(0)=q_1^{[n+1]}(0)$. The difference $d(x)=q_1(x)-q_1^{[n+1]}(x)$ satisfies
\begin{equation}\label{deq2}
\begin{split}
 &  g(x) d'(x)+[1+\alpha+iwg(x)]g'(x)d(x) \\
  &\quad\quad   =g(x)\bD\left(q_1^{[n]}(x)- q_1^{[n+1]}(x)\right)+(1+\alpha)g'(x)\left(q_1^{[n]}(x)- q_1^{[n+1]}(x) \right) \triangleq \psi^{[0]}(x),
\end{split}
\end{equation}
with zero initial condition, where $\psi^{[0]}\in C^{n}[0,a]$.
Since $\left\|\psi^{[0]}\right\|_\infty=\bO(w^{-n-1})$, we obtain from Lemma \ref{lemma01} that $\|d\|=\bO(w^{-n-1})$. Differentiating \eqref{deq2}, it follows that $d'$ satisfies
\begin{equation*}
    g(x)d''(x)+[2+\alpha+iwg(x)]g'(x)d'(x)=\psi^{[1]}(x),
\end{equation*}
with $\psi^{[1]}(x)=\bD \psi^{[0]}(x)- [(1+\alpha+iwg(x))g''(x)+iw(g'(x))^2]d(x)$. 
{\color{black}
It is clear that $\psi^{[1]}\in C^{n-1}[0,a]$ and $\left\| \psi^{[1]}\right\|_\infty=\bO(w^{-n})$, which implies that $\|d'\|_\infty=\bO(w^{-n})$.} Repeating the differential process, we have
\begin{equation}\label{dn2}
    \|\bD^md\|_\infty=\bO(w^{-n-1+m}), m=0,1,\ldots,n.
\end{equation}
{\color{black}
Combining \eqref{qn2} and \eqref{dn2}, the number $c^{[n+1]}_0$ and the solution $q_1$ meet the requirements, which finishes the proof.}
 \end{proofspecial}

\end{appendices}

\vspace{0.1cm} {\bf Acknowledgements}. The authors are grateful for the referees' helpful suggestions and insightful comments, which helped improve the manuscript significantly. The authors thank Dr. Saira and Dr. Suliman at Central South University for their careful checking of numerous details.

\end{CJK}
\end{document}